\documentclass[aap]{imsart}

\RequirePackage{amsthm,amsmath,amsfonts,amssymb}
\RequirePackage[numbers]{natbib}
\RequirePackage[colorlinks,citecolor=blue,urlcolor=blue]{hyperref}
\RequirePackage{graphicx}

\startlocaldefs
\newtheorem{thm}{Theorem}[section]
\newtheorem{lemma}[thm]{Lemma}
\newtheorem{prop}[thm]{Proposition}
\newtheorem{cor}[thm]{Corollary}
\newtheorem{conj}[thm]{Conjecture}

\newtheorem*{thm*}{Theorem}
\newtheorem*{lemma*}{Lemma}
\newtheorem*{prop*}{Proposition}
\newtheorem*{cor*}{Corollary}
\newtheorem*{conj*}{Conjecture}
\theoremstyle{remark}
\newtheorem{defn}[thm]{Definition}
\newtheorem*{defn*}{Definition}

\newtheorem{ques}[thm]{Question}
\newtheorem{rmk}[thm]{Remark}

\DeclareMathOperator{\mlt}{mlt}
\DeclareMathOperator{\wmlt}{wmlt}
\DeclareMathOperator{\gcr}{gcr}
\DeclareMathOperator{\grn}{grn}
\DeclareMathOperator{\rignum}{lrn}
\DeclareMathOperator{\rank}{rank}

\DeclareMathOperator{\tr}{Trace}
\DeclareMathOperator{\intr}{int}

\newcommand{\PP}{\mathbb{P}}
\def\prob(#1){\Phelper#1|\relax\Pchoice(#1)}
\def\Phelper#1|#2\relax{\ifx\relax#2\relax\def\Pchoice{\Pone}\else\def\Pchoice{\Ptwo}\fi}
\def\Pone(#1){\PP\left( #1 \right)}
\def\Ptwo(#1|#2){\PP\left( #1 \,\middle|\, #2 \right)}

\newcommand*{\NN}{\mathbb{N}}
\newcommand*{\RR}{\mathbb{R}}
\newcommand*{\QQ}{\mathbb{Q}}
\renewcommand{\SS}{\mathbb{S}}
\newcommand*{\MM}{\mathbb{M}}

\newcommand{\iprod}[2]{\left\langle {#1}, {#2}\right\rangle}

\usepackage{optidef}
\usepackage[shortlabels]{enumitem}

\usepackage{tikz}
\tikzstyle{vertex}=[circle, draw, inner sep=0pt,minimum size=6pt, fill=black]
\newcommand{\vertex}{\node[vertex]}
\usetikzlibrary{calc}
\tikzstyle{vertex}=[circle, draw, fill=black, inner sep=0pt, minimum size=4pt]
\tikzstyle{redvertex}=[circle, draw, red, fill=red, inner sep=0pt, minimum size=4pt]
\tikzstyle{bluevertex}=[circle, draw, cyan, fill=cyan, inner sep=0pt, minimum size=4pt]
\tikzstyle{lnode}=[circle,white,draw, inner sep=1pt, font=\scriptsize]
\tikzstyle{edge}=[line width=1pt]
\endlocaldefs

\begin{document}

\begin{frontmatter}
\title{Maximum likelihood thresholds via graph rigidity}
\runtitle{MLT via graph rigidity}

\begin{aug}
\author[daniel]{\fnms{Daniel Irving} \snm{Bernstein}},
\author[sean]{\fnms{Sean} \snm{Dewar}},
\author[shlomo]{\fnms{Steven J.} \snm{Gortler}},
\author[tony]{\fnms{Anthony} \snm{Nixon}},
\author[meera]{\fnms{Meera} \snm{Sitharam}}
\and
\author[louis]{\fnms{Louis} \snm{Theran}}
\address[daniel]{Department of Mathematics,
Tulane University}
\address[sean]{School of Mathematics, University of Bristol}
\address[shlomo]{School of Engineering 
and Applied Sciences,
Harvard University}
\address[tony]{Department of Mathematics and Statistics,
Lancaster University}
\address[meera]{Department of Computer Science,
University of Florida}
\address[louis]{School of Mathematics and Statistics,
University of St Andrews}
\end{aug}

\begin{abstract}
The maximum likelihood threshold (MLT) of a graph $G$ is the minimum number of
samples to almost surely guarantee existence of the maximum likelihood estimate
in the corresponding Gaussian graphical model. We give a new characterization of
the MLT in terms of rigidity-theoretic properties of $G$ and use this
characterization to give new combinatorial lower bounds on the MLT of any graph.

We use the new lower bounds to give high-probability guarantees
on the maximum likelihood thresholds of sparse Erd{\"o}s-R\'enyi random graphs
in terms of their average density.  These examples show that the new lower 
bounds are within a polylog factor of tight, where, on the same graph families,
all known lower bounds are trivial.

Based on computational experiments made possible by our methods,
we conjecture that the MLT of an Erd{\"o}s-R\'enyi random graph is equal to 
its generic completion rank with high probability.  Using structural 
results on rigid graphs in low dimension, we can prove the conjecture
for graphs with MLT at most $4$ and describe the threshold probability 
for the MLT to switch from $3$ to $4$.

We also give a  geometric characterization of the MLT of a 
graph in terms of a new ``lifting'' problem for frameworks that 
is interesting in its own right.  The lifting perspective 
yields a new connection between the weak MLT (where the maximum 
likelihood estimate exists only with positive probability) 
and the classical Hadwiger-Nelson problem.
\end{abstract}



\end{frontmatter}

\section{Introduction}

Modern statistical applications often require researchers to make inferences
about a large number of variables from few observations (see e.g.~\cite[Chapter
18]{hastie}). For example, certain biological network modeling problems,
including those related to gene
regulation~\cite{dobra2004sparse,schafer2005empirical,wu2003interactive} and
metabolic pathways~\cite{krumsiek2011gaussian}, can be approached by fitting a
Gaussian graphical model to a dataset that has fewer datapoints than variables.
This invites one to ask the motivating question of this paper, which was
previously explored by Uhler \cite{uhler2012geometry}, who attributes recent
interest in it to Lauritzen: \emph{given a fixed Gaussian graphical model, what
is the minimum number of datapoints so that the maximum likelihood estimate exists}? 
We now define some terms and state the question more precisely.

Let $G$ be a graph with $n$ vertices.  The Gaussian graphical model associated
with $G$ is the set of $n$-variate normal distributions $\mathcal{N}(0,\Sigma)$
so that if $ij$ is \emph{not} an edge of $G$, then $(\Sigma^{-1})_{ij} = 0$,
i.e.~the corresponding random variables are conditionally independent given all
of the other random variables. Suppose now that we have iid samples $X_1,
\ldots, X_d$ from a Gaussian graphical model.  
The MLE of the covariance is, then, the inverse of the matrix $K$ that solves the following 
optimization problem (see, e.g., \cite[p. 632]{hastie})
\begin{mini}|l|
   {K}{\tr(SK) - \log \det K}{}{\label{eq: mlt optimization}}
  \addConstraint{K \in \mathcal{S}^n_{++} \ {\rm and} \ K_{ij} = 0 \ {\rm if} \ ij \notin E(G)}
\end{mini}
where $S$ is the sample covariance\footnote{The sample covariance is $S =
\frac{1}{d}XX^T$.} and $\mathcal{S}^n_{++}$ is the set of positive definite
$n\times n$ matrices.  This is a convex problem that can be solved efficiently
in practice \cite{vandenberghe1998determinant}. Computing the MLE is a common
way to fit a Gaussian graphical model to data. If $d\ge n$ and $G$ is complete
then the MLE of the covariance is $K = S^{-1}$.  Indeed, almost surely $S^{-1}$
exists and then
\[
    \frac{d}{dK}\left(\tr(SK) - \log \det K\right) = S - K^{-1}
\]
vanishes at $S^{-1}$.  As a warmup for some of the 
ideas in Section \ref{sec: lifting}, now consider the case
when $d < n$ and $G$ is complete. Since $S$ has rank at most $d$,
 we can find a non-zero vector $v$ in the kernel of $S$.  For all 
$t\ge 0$, $I + tvv^T$ is positive definite and
\[
    \tr(S(I + tvv^T)) - \log \det (I + tvv^T)
    = \tr(S) - \log \det (I + tvv^T)\to -\infty
\]
as $t\to \infty$.  Hence 
the MLE of the covariance does not exist.  What goes wrong is that 
the sampled datapoints lie in a proper linear subspace of 
$\RR^n$, so we can ``overfit'' the sample by Gaussian densities that 
have, as their level sets, increasingly flat ellipsoids.

If $G$ is not complete, however, the MLE might exist almost surely
even when $d < n$.  This prompts the following definition.

\begin{defn}\label{def: mlt} 
The \emph{maximum likelihood threshold} (MLT) of a
graph $G$, denoted $\mlt(G)$, is the smallest number of samples\footnote{Here,
we are assuming that the samples are i.i.d.~from a distribution whose
probability measure is mutually absolutely continuous with respect to Lebesgue
measure.} required for the MLE of the Gaussian graphical model associated with
$G$ to exist almost surely.
\end{defn}

\begin{rmk}
    The maximum likelihood threshold can be similarly defined for any class of Gaussian models.
    See e.g.~\cite{amendola2021invariant,drton2019,makam2021symmetries}
\end{rmk}

\subsection{Existing bounds on the MLT}
The discussion above implies that $\mlt(K_n) = n$.
For any $G$, 
$1\le \mlt(G)\le n$, since if $H$ is a subgraph of 
$G$, then $\mlt(H) \le \mlt(G)$.
Heuristically, if $G$ is 
very sparse, we could hope that $\mlt(G)$ is much less than $n$.
However, counting edges is not enough to get good bounds, since, 
as we will see, small or sparse subgraphs can push the MLT up.

Ideally, one would like an efficient algorithm to 
compute $\mlt(G)$, but this seems difficult and the complexity 
of computing $\mlt(G)$ remains open.\footnote{It follows from Dempster's work \cite{dempster1972covariance} that one can compute $\mlt(G)$ using, e.g., cylindrical decomposition of a semi-algebraic set, but the algorithms for this task are not fast enough to be of practical interest.}
Instead, the literature, which we now review, focuses on 
finding combinatorial properties that bound the MLT, a 
problem first raised by Dempster \cite{dempster1972covariance}
and, more recently, popularized by 
Lauritzen (see \cite{uhler2012geometry,bendavid}).
The first nontrivial bounds on the MLT are due to 
Buhl \cite{buhl1993existence}.

\begin{thm}[\cite{buhl1993existence}]\label{thm: buhl}
Let $G$ be a graph with clique number $\omega(G)$ and treewidth $\tau(G)$.
Then 
\[
    \omega(G) \le \mlt(G) \le \tau(G) + 1.
\]
\end{thm}
We will see presently that both of these estimates are unsatisfactory:
computing clique number and treewidth are NP-hard problems and 
both inequalities are extremely weak.  As a running example to 
compare inequalties, we will use the complete bipartite 
graph $K_{m,m}$.  Theorem \ref{thm: buhl} implies 
that
\[
    2 = \omega(K_{m,m}) \le \mlt(K_{m,m}) \le \tau(K_{m,m}) + 1 = m+1.
\]

In a landmark paper 
that studied the (semi-) algebraic geometry of Gaussian 
maximum likelihood estimation,
Uhler \cite{uhler2012geometry} used tools from algebraic 
geometry to bound the MLT.
\begin{defn}\label{def: gcr}
Let $G$ be a graph with $n$ vertices and $m$ edges.  Let 
$\mathcal{S}^{d+1}$ be the set of symmetric matrices of rank 
$d+1$.  The \emph{generic completion rank of $G$}%
\footnote{The term ``generic completion rank'' is due to 
Blekherman and Sinn \cite{blekherman2019maximum}.}%
, denoted $\gcr(G)$, 
is the smallest 
$d+1$ so that the orthogonal projection of $\mathcal{S}^{d+1}$
onto the diagonal entries and the entries corresponding 
to the edges of $G$ is $(m + n)$-dimensional.
\end{defn}
\begin{thm}[\cite{uhler2012geometry}]\label{thm: uhler gcr mlt}
Let $G$ be a graph.  Then $\mlt(G)\le \gcr(G)$.
\end{thm}
Uhler formulated the generic completion rank 
in terms of a certain 
elimination ideal being empty, but one can compute
$\gcr(G)$ with a randomized algorithm 
and linear algebra (see \cite{gross2018maximum}).
The upper bound from Theorem~\ref{thm: uhler gcr mlt} is 
very much tighter than the one from Theorem~\ref{thm: buhl}.
It can also be used to extract other combinatorial 
bounds on the MLT, for example the presence of a $k$-core (the 
maximum induced subgraph of minimum degree $k$).
Via~\cite[Corollary~4.5]{BBS-typical}, 
Uhler's bound implies that if $k$ is the minimum integer 
such that the $k$-core of $G$ is empty, then $\mlt(G) \le k-1$.

In our running example, we have 
\[
    \mlt(K_{m,m}) \le \gcr(K_{m,m}) = m - 2
\]
(see Theorem \ref{thm: blekherman Kmn} below
for the GCR of $K_{m,m}$).  Thus, on our running 
example, Uhler's bound is better than Buhl's
\emph{and} it is much easier to compute.

For some time, it was open whether, in fact, 
$\mlt(G) = \gcr(G)$ for every graph $G$.
Blekherman and Sinn \cite{blekherman2019maximum} provided
a negative answer as part of a detailed 
study of complete bipartite graphs.  We will give a more detailed 
account of \cite{blekherman2019maximum}, but here is one summary result.
\begin{thm}[\cite{blekherman2019maximum}]\label{thm: blekherman Kmn}
Let $m, D\in \NN$ so that $m > 2$ and $D$ is largest number 
satisfying $2m > \binom{D+1}{2}$. Then
\[
    \gcr(K_{m,m}) = m\qquad \text{and}
    \qquad
    \mlt(K_{m,m}) = D.
\]
\end{thm}
Comparing with Theorem \ref{thm: buhl}, 
we see that $K_{m,m}$ has clique number $2$ and MLT $\Theta(\sqrt{m})$.  Comparing 
with Theorem \ref{thm: uhler gcr mlt}, 
we see that the upper
bound from generic completion rank is also off by an $O(\sqrt{m})$
factor, making $\gcr(G)$ far from tight as 
an upper bound.

\subsection{MLT and rigidity}
In this paper, we give new lower bounds on the MLT, 
which are more general and sharper than those mentioned above.
Our methods are based on a connection to graph rigidity theory, 
which we briefly introduce.
Figure~\ref{fig:rigidity} illustrates the following definitions for $d=2$.

\begin{defn}\label{def: framework}
Let $d\in \NN$ be a dimension.  A \emph{framework in $\mathbb{R}^d$} is 
a pair $(G,p)$ where $G$ is a graph with $n$ vertices $\{1, \ldots, n\}$ and 
$p = (p(1), \ldots, p(n))$ is a configuration of $n$ points in $\RR^d$.
Two frameworks $(G,p)$ and $(G,q)$ are \emph{equivalent} if 
\[
    \|p(j) - p(i)\| = \|q(j) - q(i)\|\qquad 
    \text{ for all edges $ij$ of $G$}
\]
and \emph{congruent} if $p$ and $q$ are related by a Euclidean isometry,
i.e.~if there exists a Euclidean isometry $T:\mathbb{R}^d\rightarrow\mathbb{R}^d$
such that $q(i)=T(p(i))$ for $i = 1,\dots,n$.
If two frameworks are congruent, then they are also equivalent but the converse need not hold.
Frameworks for which the converse \emph{does} hold are called \emph{globally  rigid} in dimension $d$,
i.e.~$(G,p)$ is globally rigid if all equivalent $d$-dimensional 
frameworks are congruent.
If this happens only for some neighborhood $U$ around $p$,
i.e.~if $(G,p)$ and $(G,q)$ are congruent whenever $q \in U$ and $(G,q)$ and $(G,p)$ are equivalent,
then $(G,p)$ is said to be \emph{rigid} in dimension $d$.
\end{defn}

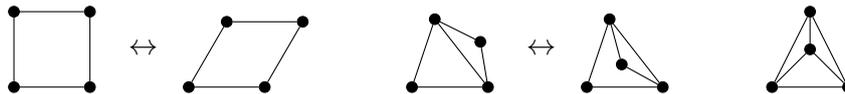
\begin{figure}[t]
    \centering
        \begin{tikzpicture}
        \vertex (1) at (0,0){};
        \vertex (2) at (1,0){};
        \vertex (3) at (1,1){};
        \vertex (4) at (0,1){};
        \path
            (1) edge (2) edge (4)
            (2) edge (3)
            (3) edge (4)
        ;
        \node at (1.7,1/2){$\leftrightarrow$};
        \vertex (1a) at (0+2.3,0){};
        \vertex (2a) at (1+2.3,0){};
        \vertex (3a) at ($ (3/2+2.3,{sqrt(3)/2}) $){};
        \vertex (4a) at ($ (1/2+2.3,{sqrt(3)/2}) $){};
        \path
            (1a) edge (2a) edge (4a)
            (2a) edge (3a)
            (3a) edge (4a)
        ;      
    \end{tikzpicture}
    \qquad\quad
    \begin{tikzpicture}
        \vertex (1c) at (0,0){};
        \vertex (2c) at (1,0){};
        \vertex (3c) at (.3,.9){};
        \vertex (4c) at (.9,.6){};
        \path
            (1c) edge (2c) edge (3c)
            (2c) edge (3c) edge (4c)
            (3c) edge (4c)
        ;
        \node at (1.7,1/2){$\leftrightarrow$};
        \vertex (1c) at (0+2.3,0){};
        \vertex (2c) at (1+2.3,0){};
        \vertex (3c) at (.3+2.3,.9){};
        \vertex (4c) at (.46+2.3,.3){};
        \path
            (1c) edge (2c) edge (3c)
            (2c) edge (3c)
            (4c) edge (3c) edge (2c)
        ;
    \end{tikzpicture}
    \qquad\quad
    \begin{tikzpicture}
        \vertex (a) at (0,0){};
        \vertex (b) at (1,0){};
        \vertex (c) at (0.5,1){};
        \vertex (d) at (0.5,0.5){};
        \path
            (a) edge (b) edge (c) edge (d)
            (b) edge (c) edge (d)
            (c) edge (d)
        ;
    \end{tikzpicture}
    \caption{Above are some frameworks in $\mathbb{R}^2$. The framework on the left fails to be rigid because there exist arbitrarily close frameworks that are equivalent but not congruent - one can deform it an arbitrarily small amount as indicated. The frameworks in the middle fail to be globally rigid since they are equivalent but not congruent. They are, however, rigid. Indeed, neither can be perturbed an infinitesimally small amount without changing edge lengths. Finally, the framework on the right is globally rigid and therefore also rigid.}
    \label{fig:rigidity}
\end{figure}

On an intuitive level, rigidity of a $d$-dimensional framework $(G,p)$ means that
if one were to physically build $G$ in $\mathbb{R}^d$ using rigid bars for the edges
and universal joints for the vertices, placed according to $p$, then the resulting structure
could not deform.
Rigidity of a specific framework is difficult to check \cite{abbot}, but 
for each dimension $d$, every graph has a generic behavior.
Following \cite{uhler2012geometry,gross2018maximum}, 
we use the following notion of 
generic, which comes from algebraic geometry.
\begin{defn}\label{def: generic}
Let $p$ be a configuration of $n$ points in $\RR^d$.  We 
say that $p$ is \emph{generic} if the coordinates of $p$ do not 
satisfy any polynomial with rational coefficients.
\end{defn}
The following theorem is fundamental in combinatorial or graph rigidity theory.
It tells us that by invoking a genericity assumption, we can treat rigidity and global rigidity
as properties of a graph rather than as properties of a framework.
\begin{thm}[\cite{asimow1978rigidity,gortler2010characterizing}]
Let $d$ be a fixed dimension and $G$ a graph.  Then either 
every generic $d$-dimensional 
framework $(G,p)$ is (globally) rigid or every generic $d$-dimensional framework
$(G,p)$ is not (globally) rigid.
\end{thm}
\begin{defn}\label{def: generic rigid / independent}
Let $d$ be a fixed dimension and $G$ be a graph with $m$ edges.  
We call $G$ \emph{(globally) $d$-rigid} if its generic $d$-dimensional
frameworks are (globally) rigid.  We call $G$ \emph{$d$-independent}
if there is an $m$-dimensional space of differential changes 
to the edge lengths of a (or any) generic framework $(G,p)$.
\end{defn}

In Figure~\ref{fig:rigidity}, the graphs underlying the frameworks in 
the middle and on the left are $2$-independent, whereas the graph of the 
framework on the right is not. To see this, note that in frameworks in 
the middle and left, it is possible to increase or decrease the length 
of any edge a small amount without changing any other edge lengths.  
This is not the case for the framework on the right.

An important fact in rigidity theory is that the  
$d$-independent graphs form the independent sets of a matroid.  Gross and 
Sullivant \cite{gross2018maximum} reformulated Theorem \ref{thm: uhler gcr mlt}
in the language of algebraic matroids (see \cite{MONTHLY} for an
introduction) and proved the following.

\begin{thm}[\cite{gross2018maximum}]\label{thm: gross sullivant}
Let $G$ be a graph.  Then the generic completion rank of 
$G$ is $d+1$ if and only if $d$ is the smallest 
dimension in which $G$ is $d$-independent.
\end{thm}
This result does not improve Uhler's upper bound on the
MLT, but it does open up the possibility 
of employing graph rigidity-theoretic ideas to understand it better.  
An interesting example is:
\begin{thm}[\cite{gross2018maximum}]\label{thm: gross sullivant planar}
If $G$ is a planar graph, then $\mlt(G) \le 4$.
\end{thm}
The proof uses the  Cauchy--Dehn--Alexandrov theorem (see \cite{gluck})
which implies that any planar graph is $3$-independent.
One can immediately deduce the same bound for the  wider 
class of $K_5$-minor free graphs using  a result of Nevo \cite{nevo}. 

Graph rigidity theory also makes it easier to compare 
treewidth to the generic completion rank.  
The following shows just how far away from tight 
Buhl's upper bound can be.  In the sequel, we will 
make statements about sequences of events $E_n$ involving 
random graph families indexed by the number of vertices $n$ that 
hold \emph{with high 
probability} (whp).  This means that $\prob(E_n) \to 1$ 
as $n\to \infty$.

It is well-known that, for $d\ge 2$, a random $(d+1)$-regular 
graph $G$ with $n$ vertices has, whp, treewidth $\tau(G) > \beta n$,
for some $\beta > 0$ (see, e.g, \cite{kloks-treewidth}).  
Since, for $d\ge 2$, the only $(d+1)$-regular 
graph that is not $d$-independent is 
$K_{d+2}$ \cite{JJ},
Theorem~\ref{thm: gross sullivant} implies that
\[
    \gcr(G) \le d + 1  < \beta n < \tau(G) + 1
\]
whp for a random $(d+1)$-regular graph $G$.

\subsection{Results and guide to reading}
In this paper, we will reformulate the MLT of a graph in terms of equilibrium
stresses, a graph rigidity-theoretic concept that plays an important role in
global rigidity. Given vertices $i$ and $j$ of a graph $G$, we write $i\sim j$
to indicate that $G$ has an edge between $i$ and $j$.
\begin{defn}\label{def: eq stress} 
    Let $G$ be a graph with $n$ vertices and let
$(G,p)$ be a framework.  An \emph{equilibrium stress} $\omega$ of $(G,p)$ is an
assignment of  weights $\omega_{ij}$ to the edges of $G$ so that, for all
vertices $i$
\[
    \sum_{j\sim i} \omega_{ij}(p(j) - p(i)) = 0 \qquad \text{(sum over neighbors $j$ of $i$).}
\]
The \emph{equilibrium stress matrix} associated to an equilibrium stress $\omega$ is the 
matrix $\Omega$ obtained by setting $\Omega_{ji} = \Omega_{ij} = -\omega_{ij}$ for all edges $ij$ of $G$,
$\Omega_{ii} = \sum_{j} \omega_{ij}$ and all other entries zero.
The \emph{rank} and \emph{signature} of $\omega$ are defined to be the rank 
and signature of $\Omega$, and $\omega$ is said to be \emph{PSD} if $\Omega$ is positive semi-definite.
\end{defn}
A fact going back to Maxwell \cite{maxwell} is that a framework $(G,p)$ in
dimension $d$ is independent if and only of it has no non-zero equilibrium
stress.  Similarly, a graph is $d$-independent   
if and only if no generic framework $(G,p)$ has a non-zero equilibrium stress.

To see the relation with Uhler's bound (Theorem \ref{thm: uhler gcr mlt}), 
we can use Theorem \ref{thm: gross sullivant} and the discussion above
to get the following formulation.
\begin{thm}[\cite{uhler2012geometry,gross2018maximum}]
Let $G$ be a graph with $n$ vertices.  Suppose that no generic framework 
in dimension $d$ supports a non-zero equilibrium stress.  Then the 
MLT of $G$ is at most $d+1$.
\end{thm}
To obtain a lower bound on the MLT, 
we will need to consider the signature of 
the equilibrium stress. Our central new tool will be the following
theorem, proved in Section \ref{sec: lifting}.
\begin{thm}\label{thm: main mlt stress}
Let $G$ be a graph with $n$ vertices. Then the MLT of $G$ is
$d+1$ if and only if $d$ is the smallest dimension in which 
no generic $d$-dimensional framework supports a non-zero 
PSD equilibrium stress.
\end{thm}
This technical theorem along with some known 
graph rigidity-theoretic results and arguments will allow us to 
significantly expand our understanding of the MLT (as well as 
directly reproduce most of what is already understood).

To return to our running example, Theorem \ref{thm: blekherman Kmn}
implies that $K_{m,m}$ generically supports equilibrium stresses
in dimension $m-2$, but that they are all indefinite.  In 
Section \ref{sec: complete bipartite}, we will re-derive 
Theorem \ref{thm: blekherman Kmn} by first understanding 
the equilibrium stresses of complete bipartite graphs.

There is a geometric counterpart to Theorem \ref{thm: main mlt stress},
originally conjectured by Gross and Sullivant \cite{gross2018maximum}. A
$d$-dimensional framework $(G,p)$ \emph{has full affine span} if $p $ affinely
spans $\mathbb{R}^d$.
\begin{thm}\label{thm: main mlt lift}
Let $G$ be a graph with $n$ vertices. Then the MLT of $G$ is 
$d+1$ if and only if $d$ is the smallest dimension in which 
every generic $d$-dimensional framework $(G,p)$
is equivalent to an $(n-1)$-dimensional framework 
$(G,\tilde{p})$ with full affine span.
\end{thm}

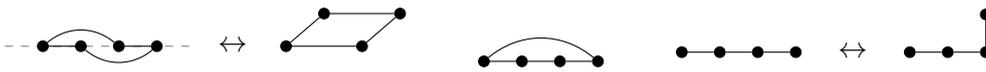
\begin{figure}[t]
    \centering
    \begin{tikzpicture}
        \node at (0,0.5){};
        \draw[gray,dashed] (-0.5,0)--(2,0);
        \vertex (1b) at (0,0){};
        \vertex (2b) at (0.5,0){};
        \vertex (3b) at (1,0){};
        \vertex (4b) at (1.5,0){};
        \path
            (1b) edge[bend left = 40] (3b)
            (2b) edge[bend right = 40] (4b)
            (3b) edge (4b)
            (1b) edge (2b)
        ;
        \node at (2.5,0){$\leftrightarrow$};
        \vertex (1a) at (0+3.2,0){};
        \vertex (2a) at (1+3.2,0){};
        \vertex (3a) at ($ (3/2+3.2,{0.5*sqrt(3)/2}) $){};
        \vertex (4a) at ($ (1/2+3.2,{0.5*sqrt(3)/2}) $){};
        \path
            (1a) edge (2a) edge (4a)
            (2a) edge (3a)
            (3a) edge (4a)
        ;
    \end{tikzpicture}
    \qquad
    \begin{tikzpicture}
        \node at (0,0.5){};
        \vertex (1a) at (0,0){};
        \vertex (2a) at (0.5,0){};
        \vertex (3a) at (1,0){};
        \vertex (4a) at (1.5,0){};
        \path
            (1a) edge (2a)
            (2a) edge (3a)
            (3a) edge (4a)
            (1a) edge[bend left=40] (4a)
        ;
    \end{tikzpicture}
    \qquad
    \begin{tikzpicture}
        \node at (0,0.5){};
        \vertex (1a) at (0,0){};
        \vertex (2a) at (0.5,0){};
        \vertex (3a) at (1,0){};
        \vertex (4a) at (1.5,0){};
        \path
            (1a) edge (2a)
            (2a) edge (3a)
            (3a) edge (4a)
        ;
        \node at (2.25,0){$\leftrightarrow$};
        \vertex (1a) at (3+0,0){};
        \vertex (2a) at (3+0.5,0){};
        \vertex (3a) at (3+1,0){};
        \vertex (4a) at (3+1,0.5){};
        \path
            (1a) edge (2a)
            (2a) edge (3a)
            (3a) edge (4a)
        ;
    \end{tikzpicture}
    \caption{The framework in $\mathbb{R}^1$ on the left is equivalent to a framework in 
    $\mathbb{R}^3$ with full-dimensional affine span. To see this, first note that it 
    is equivalent to the framework in $\mathbb{R}^2$ to the right of it. Then note 
    that this two-dimensional framework is equivalent to a framework in $\mathbb{R}^3$ 
    with full-dimensional affine span since we can lift one of the vertices into the 
    third dimension without changing edge-lengths. However, the maximum likelihood 
    threshold of the underlying graph, the four-cycle, is \emph{not} two since every 
    framework equivalent to the framework in the middle has a one-dimensional affine span.
    On the other hand, the path with four vertices has an MLT of $2$ because 
    any generic one-dimensional framework on it can be folded out to three dimensions. 
    On the right, we see such a one-dimensional framework folding out into two 
    dimensions. We can further fold it into three by bringing the vertex on the 
    left out of the affine plane spanned by the other vertices.
    }
    \label{fig: lifting}
\end{figure}

See Figure~\ref{fig: lifting} for an illustration of Theorem~\ref{thm: main mlt lift}.

To give combinatorial bounds, we use a connection between 
globally rigid graphs and PSD equilibrium stresses from \cite{cgt1}.  
Our bounds are in terms of a new graph parameter based on globally 
rigid subgraphs of $G$.
\begin{defn}\label{def: grn}
    The \emph{global rigidity number} of $G$, denoted $\grn(G)$, is the maximum $d$ such 
    that $G$ is globally $d$-rigid and has at least
    $d+2$ vertices.
    The \emph{globally rigid subgraph number} of $G$, denoted $\grn^*(G)$, is the maximum 
    $d$ so that $G$ contains a subgraph $H$ on at least $d+2$ vertices that 
    is   globally rigid.
\end{defn}
We obtain the following new lower bound on the MLT of a 
graph.
\begin{thm}\label{thm: grn mlt}
Let $G$ be a graph.  Then $\grn(G) + 2\le\grn^*(G) + 2\le \mlt(G)$.
\end{thm}
Since complete graphs are globally $d$-rigid for all $d$, Theorem \ref{thm: grn mlt} 
generalizes the lower 
bound of Theorem \ref{thm: buhl}.  
To our knowledge, this is the 
first unconditional improvement of Buhl's lower bound  
from 1993 (Theorem \ref{thm: buhl}).

In our running example of $K_{m,m}$, the 
lower bound from Theorem \ref{thm: grn mlt} 
gives the right answer.  The complete bipartite 
graph $K_{m,m}$ has global rigidity number $m-3$
\cite{CGTunpublished} and, by Theorem \ref{thm: blekherman Kmn}, 
MLT $m-1$.  Hence $\mlt(K_{m,m}) = \grn(K_{m,m}) + 2$.

In Section \ref{sec: mlt 3}, we combine Theorem \ref{thm: grn mlt}
with results on graph rigidity in dimension $2$ to completely
solve the MLT problem for small values of $\mlt(G)$ and $\gcr(G)$.
The main result (which is best possible -- see Remark~\ref{rmk: best possible}) is as follows.
\begin{thm}\label{thm: equality of gcr and mlt}
If $G$ is a graph and $\mlt(G) \le 3$ or $\gcr(G) \le 4$, then $\mlt(G) = \gcr(G)$.
\end{thm}

Theorem \ref{thm: grn mlt} is also strong enough to give a 
quick proof of the results in \cite{blekherman2019maximum}.
Section \ref{sec: complete bipartite} explores the connection.

\subsection{Improvement in MLT bounds}
The following table summarizes the MLT bounds for various families of graphs given in 
this paper and compares with the best known previous bounds.
Let $d>1$.
\begin{center}
\begin{tabular}{ |c|c|c| } 
 \hline
Type of Graph & Best Previous MLT interval & This Paper \\ 
\hline
 Minimally $d$-rigid & $[2,d+1]$ Thm \ref{thm: buhl} & $d+1$ Cor \ref{cor:min rigid implies gcr=mlt}\\ 
\hline
 Globally $d$-rigid $d$-circuit& $[2,d+2]$ Thm \ref{thm: buhl} & $d+2$ Cor \ref{cor:gr+cir} \\ 
\hline
$G(n, c/n)$ $0<c<c’_2$ & $[3,4]$ Thm \ref{thm: uhler gcr mlt} \cite{BBS-typical} & 3 whp Thm \ref{thm: random mlt 3 4}\\ 
\hline 
$G(n, c/n)$ $c’_2<c<c_3$& $[3,4]$ Thm \ref{thm: uhler gcr mlt} \cite{BBS-typical} & 4 whp Thm \ref{thm: random mlt 3 4}\\
\hline
$G(n,M_d\log n/n)$ & $\exists k [3, (\log n)^k]$ whp Thm \ref{thm: buhl},\ref{thm: uhler gcr mlt}  & 
$\exists k [d, (\log n)^k]$ whp Cor \ref{cor: rand mlt}\\
\hline
\end{tabular}
\end{center}

\section{Examples and conjectures}
The families of graphs for which the MLT has 
been computed exactly are quite limited in the 
literature.  
As mentioned in the introduction, Buhl \cite{buhl1993existence}
computes $\mlt(K_{d+2}) = d + 1$ and Bleckhermann and Sinn 
\cite{blekherman2019maximum} compute $\mlt(K_{m,n})$ (see 
Section \ref{sec: complete bipartite}).  Uhler \cite{uhler2012geometry}
provides, in addition, that the MLT of a cycle is $3$.

This sections contains examples illustrating improvements 
that can be obtained from our methods and some remaining conjectures.

\subsection{4-regular graphs} As discussed above, 
any connected $4$-regular graph $G$ on $n > 5$ vertices 
is $3$-independent and, hence, 
by Theorem \ref{thm: gross sullivant}, $\gcr(G) \le 4$.
Since $G$ has $2n > 2n - 3$ 
edges, the Laman--Pollaczek-Geiringer Theorem (Theorem 
\ref{thm: LPG}, below) implies that $G$ is $2$-dependent so  
$\gcr(G) > 3$,
again via Theorem \ref{thm: gross sullivant}.  On the other hand, it is 
well-known that, whp, a random $4$-regular graph has clique number 
at most $3$\footnote{One way to see this is that a $4$-regular graph containing 
a subgraph isomorphic to $K_4$ is not essentially $6$-connected.  Random 
$4$-regular graphs are essentially $6$-connected with high 
probability \cite{wormald}.}.
We now have that, whp, for a random $4$-regular graph $G$:
\[
    \omega(G) \le 3 < 4 = \mlt(G). 
\]
This gap, while not terribly impressive, was easy to get.

\subsection{Examples from the applied literature}
We now consider several examples of Gaussian graphical models that appear 
in the applied literature on gene networks,
a setting where the models are typically high-dimensional.
Our rigidity theoretic 
methods make it easy to compute the MLT for these instances, to the 
extent that the work is simple enough to do by hand in many cases.

The approach is to show, first, that a graph $G$ is $d$-independent by iteratively 
removing vertices of degree at most $d$ until there are no vertices left.  Such a 
graph cannot have a general position framework with a non-zero equilibrium stress
in dimension $d$, so it must be $d$-independent.
For the matching 
lower bound, we find a globally $(d-1)$-rigid subgraph with at least 
$d+1$ vertices.  In cases where the globally rigid subgraph is complete, 
one could get the lower bound from the existing Theorem \ref{thm: buhl},
but our approach gives a unified treatment of both bounds.  We now 
turn to specific examples.



\begin{figure}
    \centering
    \includegraphics[width = 0.9\textwidth]{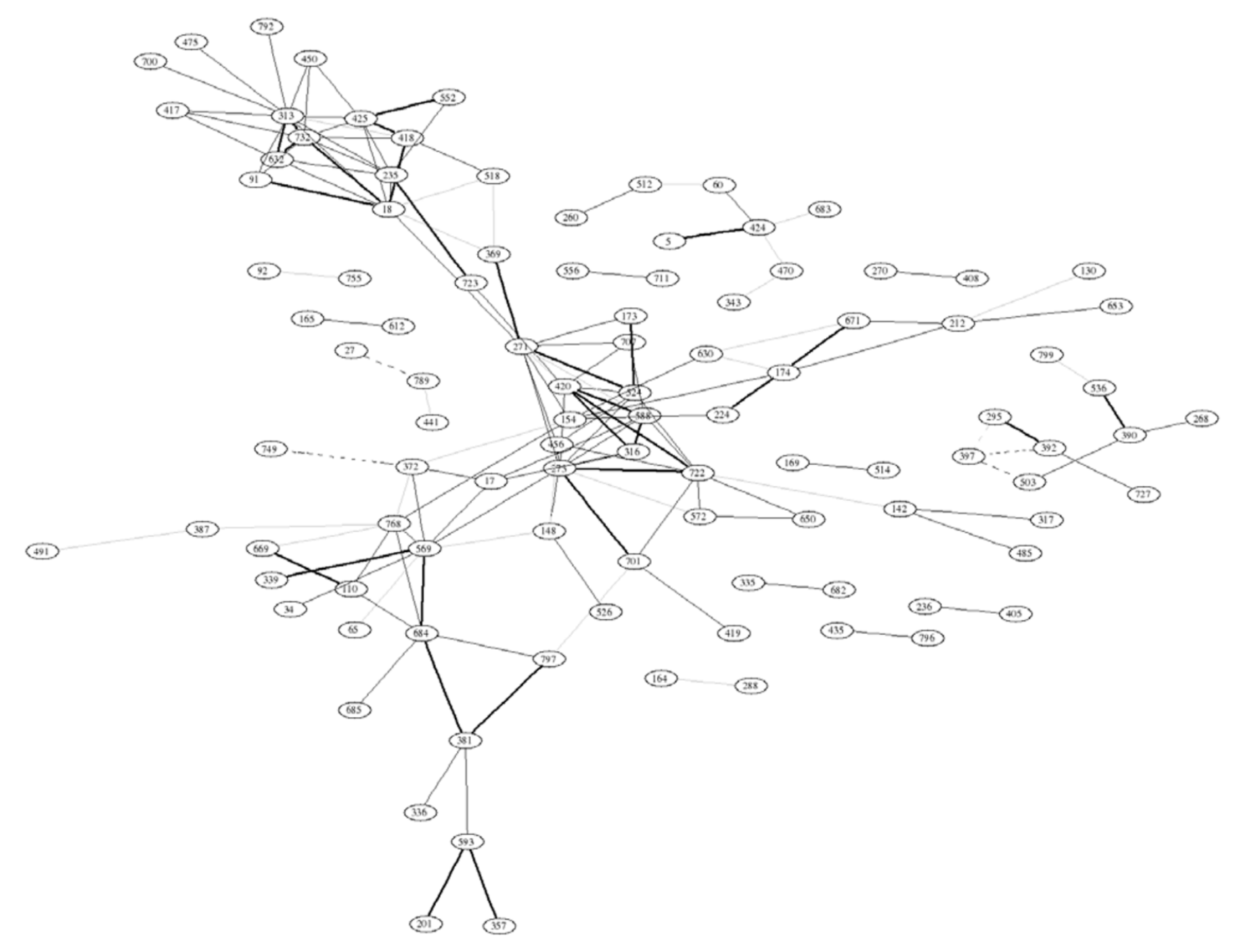}
    \caption{This graph is reproduced from \cite[Figure 1]{opgen2007correlation}.  Although it has vertices of 
    degree greater than four, the graph can be reduced to a graph with no vertices by iteratively finding a vertex of degree 
    at most four and removing it and any incident edges.  For example, the degree of vertex $174$ is five, but once its 
    neighbor $671$, which has degree three, is removed, the degree of $174$ decreases, allowing for its removal.  
    Continuing in this manner eventually removes every vertex.}
    \label{fig: simmer 07 graph}
\end{figure}
In \cite[Figure~1]{opgen2007correlation} there is a graph $G$ on 150 edges inferred from Arabidopsis thaliana data.
By iteratively deleting vertices of degree at most 4 every vertex of the graph is deleted, 
implying that $G$ is $4$-independent.
Theorems~\ref{thm: gross sullivant} and~\ref{thm: uhler gcr mlt} therefore imply $\mlt(G) \le 5$.
Since $G$ contains a $K_5$ subgraph on the vertices labeled 313,425,732,235,18,  Theorem~\ref{thm: buhl}
implies moreover that $\mlt(G) \ge 5$, and so we have equality.

Similarly, \cite[Figure~1]{schafer2005empirical} presents a simulated gene association network $G$ with 100 vertices.
By iteratively deleting vertices of degree at most 2 every vertex of the graph is deleted, so $G$ is 
$2$-independent.
Since $G$ is not cycle-free, $G$ is not 1-independent so Theorem~\ref{thm: gross sullivant}
implies that $\gcr(G) = 3$.  Theorem~\ref{thm: equality of gcr and mlt} therefore implies $\mlt(G) = 3$.

Finally, \cite[Figure~5(a)]{schafer2005shrinkage}
gives another gene network $G$ inferred from E. coli data,
this time by a so-called shrinkage Gaussian graphical model approach.
Iteratively deleting vertices of degree until there are no vertices left 
certifies that that $G$ is 3-independent.
Again, Theorems~\ref{thm: gross sullivant} and~\ref{thm: uhler gcr mlt} imply $\mlt(G) \le 4$,
and since $G$ has a $K_4$ (on b1583,lacA,yaeM,lacZ), Theorem~\ref{thm: buhl} implies $\mlt(G) \ge 4$, giving equality.

While these illustrative examples are handled by our bounds and  
a simple heuristic algorithm, we stress that many graphs which are $d$-independent
do not have empty $(d+1)$-core and that there is no a priori reason to expect a globally 
$(d-1)$-rigid subgraph of a large sparse graph to be small or easy to spot.  We investigate 
random graph families that have this behavior next.

\subsection{Random graphs near the 2-dimensional rigidity transition}
A very sparse Erd\"os-R\'enyi random graph $G(n,c/n)$ has each of the 
$\binom{n}{2}$ possible edges independently with probability $c/n$.
Hence the vertex degrees have (dependent) binomial distributions 
with parameters $n-1$ and $c/n$.  As $n\to \infty$ the 
vertex degrees approach  Poisson random variables with 
parameter $c$.  Hence, when discussing such graphs we will refer to 
$c$ as the ``expected average degree''.

A central result in the theory of random graphs describes the emergence 
and growth of the $k$-core of an Erd\"os-Rényi random graph $G(n,p)$.
We state a simplified version (using the letter $d$ instead of $k$, 
because it makes more sense for our application).
\begin{thm}[\cite{PSW}]\label{thm: kcore}
For each $d\ge 2$ there are constants $c_d < c'_d$, so that 
\begin{itemize}
    \item If $c < c_d$, then, whp, $G(n,c/n)$ has an empty $(d+1)$-core.  If 
         $c > c_d$, then, whp, $G(n,c/n)$ has a non-empty $(d+1)$-core spanning 
        $\Omega(n)$ vertices.
    \item If $c_d < c < c'_d$, then, whp, the $(d+1)$-core of $G(n,c/n)$ has average 
        degree lower than $2d$.  If $c > c'_d$, then, whp, the $(d+1)$-core 
        has average degree at least $2d$.
\end{itemize}
Moreover, $c'_d < c_{d+1}$.
\end{thm}
For reference, the value of $c_2\approx 3.35$ and $c'_2\approx 3.59$.  As 
$d$ becomes large, $c'_d$ approaches $2d$.

What is most important for is that $c'_2/n$ is the threshold function for 
an Erd\"os-Rényi random graph to be $2$-independent.
\begin{thm}[\cite{KMT}]
In the notation of Theorem \ref{thm: kcore}, if $c < c'_2$, then, whp, 
$G(n,c/n)$ is $2$-independent\footnote{And, in fact, has no rigid component 
spanning more than $3$ vertices.}.  If $c > c'_2$, then, whp, $G(n,c/n)$
contains a globally $2$-rigid\footnote{That we have global 
rigidity was first noted by Bill Jackson.}
subgraph spanning a $(1-o(1))$-fraction of the $3$-core (hence spanning 
$\Omega(n)$ vertices).
\end{thm}
Using Theorem \ref{thm: equality of gcr and mlt}, we get 
immediately:
\begin{thm}\label{thm: random mlt 3 4}
In the notation of Theorem \ref{thm: kcore}, 
\begin{itemize}
    \item If $c_2 < c < c'_2$, then, whp, $\mlt(G(n,c/n)) = 3$.
    \item If $c'_2 < c < c_3$, then, whp $\mlt(G(n,c/n)) = 4$.
\end{itemize}
In particular, in this range, $\mlt(G(n,c/n)) = \gcr(G(n,c/n))$, whp.
\end{thm}
As a comparison, since Erd\"os-Rényi graphs with 
$p = c/n$, for any $c > 0$, 
are well-known to have clique number $3$, whp, 
Theorem \ref{thm: buhl} would give a lower bound of $3$
across this entire range.  Combining Theorem \ref{thm: uhler gcr mlt}
and \cite[Corollary~4.5]{BBS-typical},  we would get an 
upper bound of $4$ on the MLT in the range $c_2 < c <c'_2$.
Using our methods, we get an exact result in the whole range.

\subsection{Conjectures and a question}
The following conjecture generalizes Theorem 
\ref{thm: random mlt 3 4}.
\begin{conj}\label{conj: random rigid}
Let $c > 0$ be fixed.  Then, $\mlt(G(n,c/n)) = \gcr(n,c/n)$, whp.
\end{conj}
In other words, we conjecture that for \emph{all} very sparse 
Erd\"os-Rényi random graphs, the MLT and GCR coincide whp.

We will go further and make a structural conjecture that goes 
beyond Theorem \ref{thm: random mlt 3 4} even for GCR $4$.  We need 
some terminology.  A edge $ij$ of a graph $G$ is \emph{$d$-redundant} 
if there is some generic $d$-dimensional framework $(G,p)$ 
that has an equilibrium stress with a non-zero coefficient on the 
edge $ij$.  The \emph{$d$-redundant subgraph of $G$} is the 
subgraph comprising all the $d$-redundant edges.
\begin{conj}\label{conj: random stressable subgraph}
Let $c > 0$ be fixed.  Then, whp, if $G = G(n,c/n)$ has 
GCR $d+1 \ge 4$, the $(d-1)$-redundant subgraph of 
$G$ is generically globally $(d-1)$-rigid.
\end{conj}
This conjecture implies that the MLT and GCR are equal, whp, 
using Theorem \ref{thm: grn mlt}, but with the added 
precision of identifying the globally rigid subgraph that 
certifies the lower bound.  
The conjecture does not 
include GCR $3$ and lower, where we do not expect it 
to be true.  The reason is that, for $c < c'_2$, 
the $1$-redundant subgraph is not $2$-connected 
whp, and $2$-connectivity is necessary for 
generic global rigidity \cite{hendrickson1992}.
We provide theoretical and experimental evidence for the 
conjecture in the next section.  

We conclude with a question, which does not 
seem empirically resolved by our experiments.
\begin{ques}\label{ques: random stressable core}
Let $c > 0$ be fixed.  Is it true that, whp, if 
$G = G(n,c/n)$ has MLT $d + 1 \ge 4$ that the $(d-1)$-redundant 
subgraph of $G$ is exactly the $d$-core?
\end{ques}
A positive answer to the question would imply that the MLT (and also 
global rigidity) of sparse Erd\"os-Rényi random graphs has the 
same evolution as the $d$-cores.

\subsection{Theoretical evidence for the conjectures}
Aside from it being true for $d=2$, a weaker 
result for $d\ge 3$ provides evidence for the 
conjecture.  In Appendix \ref{sec:app:rand} we will show the 
following result\footnote{While this paper was in preparation, 
Lew, Nevo, Peled and Raz \cite{lew2022} found the sharp threshold 
for rigidity.  However, their improved result does not lead to a qualitatively 
better bound on the MLT than we state here.}.
\begin{thm}\label{thm: rand rigid 2}
Let $d\ge 1$.  There is a $C_d > 0$ and a $k\in \NN$ such that, whp,\\
$G(n,C_d(\log n)^k/n)$ is $d$-rigid.
\end{thm}
We now get, by combining Theorem \ref{thm: rand rigid 2} and 
Theorem \ref{thm: rnr mlt} below:
\begin{thm}\label{thm: rand mlt lower}
Let $d\ge 1$.  There is a $M_d > 0$ and $k\in \NN$, such that, whp,\\
$\mlt(G(n,M_d(\log n)^k/n)) \ge d$.
\end{thm}
Since, whp, the graphs in Theorem \ref{thm: rand mlt lower} have
clique number $3$, the lower bound from Theorem \ref{thm: buhl}
becomes increasingly ineffective.  On the other side, we do not 
have a good upper bound on $\gcr(G(n,M_d(\log n)^k/n))$.
A simple bound on the GCR is based on maximum degree.
If a graph $G$ has maximum degree $\Delta$
then it is $\Delta$-independent, so we can use the fact that, 
 whp, the maximum degree of $G(n,M_d(\log n)^k/n)$ is 
 $O(\operatorname{polylog}(n))$ to deduce:
\begin{cor}\label{cor: rand mlt}
In the notation of Theorem \ref{thm: rand mlt lower}, then, 
there are  $k,\ell \in \NN$ such that, whp, for every $d\ge 1$,
\[
    d\le \mlt(G(n,M_d(\log n)^k/n)) \le (\log n)^\ell
\]
where $M_d$ and $k$ are from Theorem \ref{thm: rand mlt lower}.
\end{cor}
We have not tried to optimize the upper bound, since we believe 
that the maximum degree is not a good estimate of the 
GCR for these random graphs.

\subsection{Experimental evidence for the conjectures}
To test Conjectures \ref{conj: random rigid} and 
\ref{conj: random stressable subgraph}, we ran experiments
on Erd\"os-Rényi random graphs.
For fixed numbers of vertices $n=30,500$ and expected average 
degree $c$ from $3$ to $30$ with step size $0.1$ we generated
$20$ samples from $G(n,c/n)$ and computed: the GCR $d+1$; 
whether the $(d-1)$-redundant subgraph is generically 
globally rigid in dimension $d$; and whether the $(d-1)$-redundant 
subgraph is equal to the $d$-core.

Computing the GCR, identifying the redundant subgraph, and 
checking global rigidity are done using a randomized 
algorithm based on linear algebra over finite fields
\cite{gortler2010characterizing}.  

These computations rely on the results of this paper in an essential 
way.  In particular, we need the lower bound from 
Theorem \ref{thm: grn mlt} to certify that the MLT is 
equal to the GCR.
Without Theorem \ref{thm: grn mlt}, we would have to rely 
on heuristic numerical experiments instead.

In all of our runs, when the GCR is at least $4$, the 
$(d-1)$-redundant subgraph was generically globally rigid.
This is consistent with Conjectures~\ref{conj: random stressable subgraph}
and~\ref{conj: random rigid}.
(Recall that for GCR of $3$, 
Conjecture~\ref{conj: random rigid} is true 
from Theorem~\ref{thm: random mlt 3 4}, and that 
Conjecture~\ref{conj: random stressable subgraph}
is not expected to hold.)

The charts in Figure \ref{fig: charts} show the evolution 
of the GCR (which is equal to the MLT) as the expected average degree increases in our experiments for two values of $n$ 
over a common range of $c$.
The transition from $3$ to $4$ happens quickly between $c=3.4$ and 
$c=3.6$ as predicted by Theorem \ref{thm: random mlt 3 4}.

In all of our runs, when the GCR is in the range  $3-6$, the 
equality 
of the $(d-1)$-redundant subgraph and the $d$-core held 
at least $97\%$ of the time, but this fraction did not seem to 
go up as we increased the number of vertices from $30$ to $1000$.
In all of our runs, when the GCR is at least $7$, the 
equality 
of the $(d-1)$-redundant subgraph and the $d$-core held 
all of the time.

\begin{figure}[ht]
\centering 
    \begin{center} \begin{tabular}{cccc}
  {  \includegraphics[width=3in]{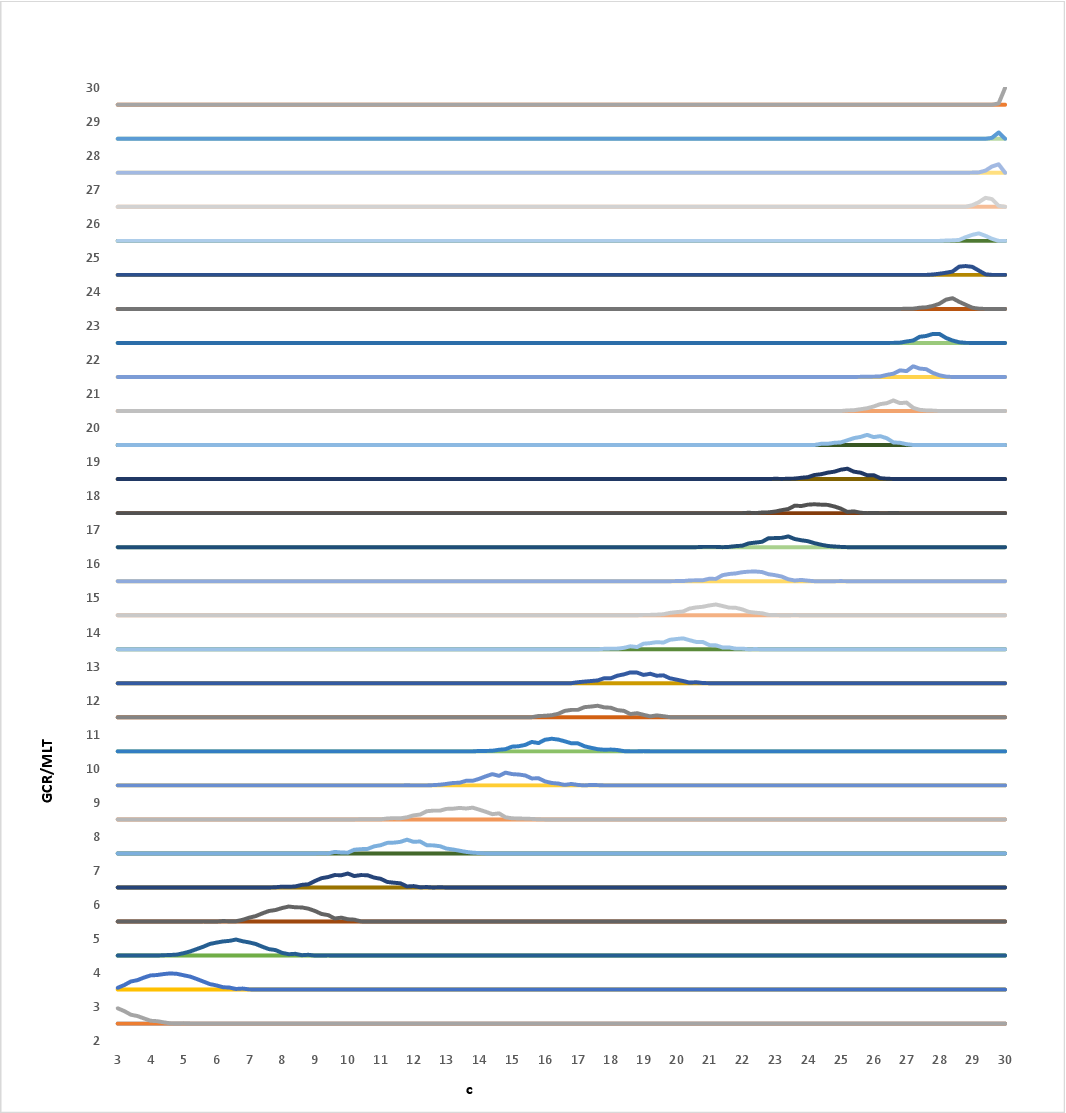}}&
  {  \includegraphics[width=3in]{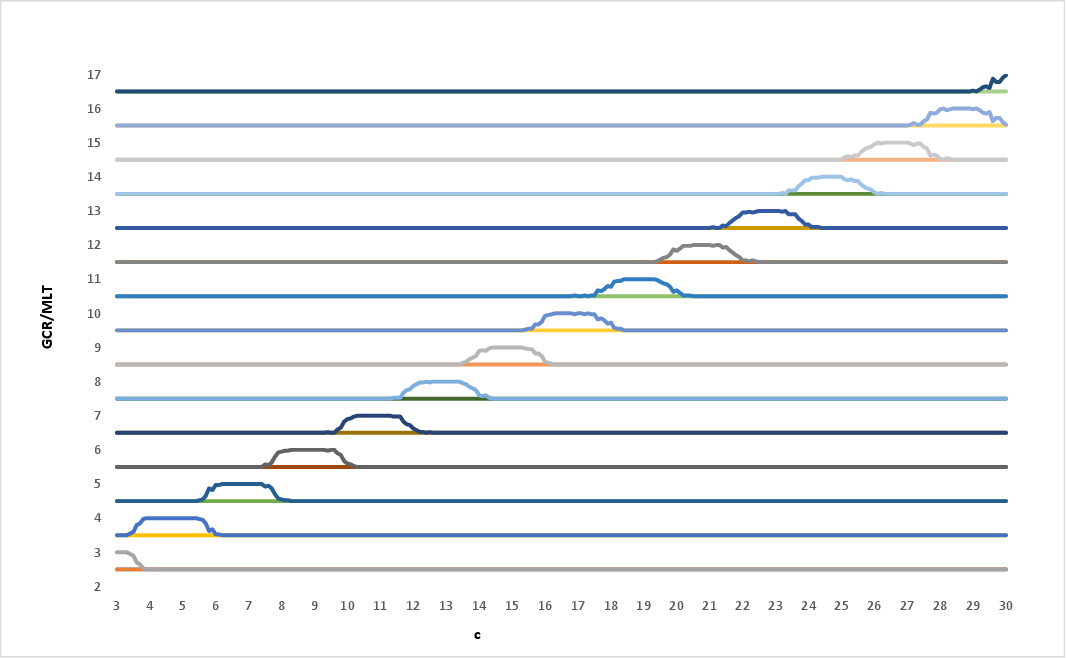}}
\end{tabular} 
\end{center}
\caption{Evolution of the GCR in $G(n,c/n)$ as the expected average
degree, $c$ increases.  The $x$-axis shows the $c$ 
values. The $y$-axis is labeled by the integer valued,
possible GCR values.
As stated in the text, all of the given GCR values were equal to the MLT.
Each curve in horizontal slice  
shows the proportion of samples that produced the associated GCR.
The left chart is $n=30$ and the
right is $n=500$.
Already, in the left chart, we see that each GCR values is only
observed for a small range of $c$ values. 
The left chart spans the
range from sparse to dense graphs on $30$ vertices.
In the right chart, we see that, for most $c$ under consideration, 
we only observe
a single GCR value, and that these phase transition regions,
especially in the lowest $c$ ranges, 
have tightened. This is consistent with our conjectures.
The right chart only looks at relatively 
sparse graphs on $500$ vertices. Also note, in this sparse
region, that the  trend of the peaks of the bumps is linear, with 
GCR approximately equal to $3/2+c/2=3/2+m/n$, where $m$ is the expected
number of edges.
}
\label{fig: charts}
\end{figure}
The code is available from the author's web site\footnote{\url{https://gist.github.com/theran/994b4d355e56529f5e6642fec4aead98}}
and the data used to generate the charts, including the random graphs, 
is available upon request.

\section{Stress geometry of the MLT}\label{sec: lifting}
In this section we develop a detailed geometric understanding
of the MLT.  Our main tool for doing this is the theory 
of PSD equilibrium stresses of frameworks. The 
importance of PSD equilibrium stresses has long 
been known in rigidity \cite{c1} 
and graph theory \cite{tutte-howtodraw,LLW}.
Uhler \cite{uhler2012geometry} has pointed out the 
semi-algebraic nature of the MLT problem.  Here we make 
the connection precise enough to exactly describe the MLT
in terms of equilibrium stresses.

\subsection{Linear equilibrium stresses}
To connect to the optimization problem underlying the 
MLT, we introduce the notion of a linear equilibrium 
stress, which is implicit in a number of works around 
rigidity in geometries with projective models (see
\cite{NW-metricchange} and the references therein).
We start with some notation relating to vector configurations.

\begin{defn}\label{def: flattenable}
Let $q$ be a configuration of $n$ vectors in $\RR^{d+1}$.
Denote by $t_i$ the last coordinate of $q(i)$ and by $Q$
the $(d+1)\times n$ matrix with the $q(i)$ as its columns.
We say that $q$ is \emph{flat} if all the $t_i$ are one, 
and that $q$ is \emph{flattenable} if all the $t_i$ 
are non-zero.
\end{defn}
Generic configurations are clearly flattenable.
There is a unique flat configuration associated 
with a flattenable configuration $q$ arising from scaling 
$q(i)$ by $1/t_i$.
Flat vector configurations in $\mathbb{R}^{d+1}$ are naturally associated with 
affine point configurations in $\mathbb{R}^d$.

\begin{defn}\label{def: std lifting}
Let $p$ be a configuration of $n$ points in $\RR^d$.
We denote by $\hat p$, the vector configuration 
in $\RR^{d+1}$ defined by the standard homogeneous 
coordinates for $p$, i.e.
\[
    \hat p(i) = \begin{pmatrix} p \\ 1\end{pmatrix}.
\]
The matrix $\hat P$ is $(d+1)\times n$ with the 
vectors $\hat p(i)$ as its columns.
\end{defn}

\begin{defn}\label{def: linear stress}
Let $d$ be a dimension. Let $G$ be a graph with $n$ vertices and 
let $q$ be a vector configuration of $n$ 
vectors in $\RR^{d+1}$.
An assignment $\omega$ of weights $\omega_{ij}$ to the 
edges $ij$ of $G$ and $\omega_{ii}$ to the vertices 
of $G$ is a \emph{linear equilibrium stress}
for $q$ if 
\begin{equation}\label{eq:linearStress}
   \sum_{j\sim i} \omega_{ij}q(j) = \omega_{ii}q(i) 
    \qquad 
    \text{(all $i\in V(G)$).}
\end{equation}
For a fixed $\omega$, we say that
\emph{$q$ satisfies $\omega$} if \eqref{eq:linearStress} holds.
A \emph{linear equilibrium stress matrix} $\Omega$ for
$q$ is a symmetric $n$-by-$n$ matrix with
$\Omega_{ij} = 0$ for non-edges of $G$ such that
\[
    \Omega  Q^T = 0,
\]
where $Q$ is the $(d+1)\times n$ matrix with the 
$q(i)$ as its columns.
Given a linear equilibrium stress $\omega$ for $q$,
we can
make a linear equilibrium stress matrix for it
by setting $\Omega_{ij}=\Omega_{ji}=-\omega_{ij}$
on the edges
and setting the diagonals $\Omega_{ii}=\omega_{ii}$.
Hence the vector configurations satisfying a given set of 
weights arise from the kernel of the associated linear 
equilibrium stress matrix.
\end{defn}

The following lemma is immediate. It gives the precise relationship between equilibrium stresses and \emph{linear} equilibrium stresses.

\begin{lemma}\label{lem: affine to linear stress}
Let $G$ be a graph with $n$ vertices and let $(G,p)$
be a $d$-dimensional framework.  Then for any 
equilibrium stress $\omega$ of $(G,p)$, the 
associated stress matrix gives a linear equilibrium 
stress of $\hat p$.  Any linear equilibrium stress 
matrix $\Omega$ for $\hat p$ is also an equilibrium 
stress matrix for $(G,p)$.
\end{lemma}

Linear equilibrium stresses are well-behaved 
under scaling. Results similar to the following can be 
found in e.g.~\cite{cgt2,connelly2010global}.

\begin{lemma}\label{lem: linear stress scale}
Let $G$ be a graph with $n$ vertices and let $q$ be a vector configuration in $\RR^{d+1}$.  If $\Omega$
is a linear equilibrium stress matrix for $q$
and $s_1, \ldots, s_n$ are any non-zero real 
numbers, then the configuration $\tilde q$, defined
by 
\[
    \tilde q(i) = \frac{1}{s_i}q(i)
\]
has a linear equilibrium stress matrix with the 
same signature as $\Omega$.
\end{lemma}

\begin{proof}
Take $q$ and the $s_i$ as in the statement, and let 
$\omega$ be the linear equilibrium stress for $q$
from the statement. For each vertex $i$ and edge $ij$, define
\[
    \tilde\omega_{ij} = s_i s_j \omega_{ij}
    \qquad 
    \text{and}
    \qquad 
    \tilde\omega_{ii} = s^2_i\omega_{ii}.
\]
Then $\tilde\omega$ is a linear equilibrium stress for $\tilde{q}$ because for each vertex $i$ we have
\[
    \sum_{j\sim i} \tilde\omega_{ij}\tilde q(j) 
    = s_i  \sum_{j\sim i} \omega_{ij}q(j)
    = s_i \omega_{ii}q(i) = s_i^2 \omega_{ii}\tilde q(i)
    = \tilde\omega_{ii} q(i).
\]
Let $\tilde{\Omega}$ be the stress matrix associated to $\omega$
and let $S$ be the diagonal matrix whose diagonal entries are $s_1,\dots,s_n$.
Then $\tilde{\Omega}=S\Omega S$ and thus $\Omega$ and $\tilde{\Omega}$ have the same signature.
\end{proof}

We get an important special case when $s_i$ is the last coordinate of $q(i)$ for each $i$.

\begin{lemma}\label{lem: flattening stress}
Let $G$ be a graph with $n$ vertices, let $q$ be a
flattenable configuration of $n$ vectors in $\RR^{d+1}$, and let 
$(G,p)$ be the framework in $\RR^d$ that arises from 
flattening $q$ and deleting the all-ones coordinate.
If there is a linear equilibrium stress matrix $\Omega$
for $q$, then $p$ has an equilibrium stress matrix of the same signature as $\Omega$.
\end{lemma}

\begin{proof}
If we denote by $\hat p$ the flattening of $q$, then 
by Lemma \ref{lem: linear stress scale}
there is a linear equilibrium stress for $\hat p$ with 
the same signature as $\Omega$.  This stress is
an equilibrium stress of $(G,p)$ by Lemma 
\ref{lem: affine to linear stress}.
\end{proof}

\subsection{The optimization problem}
We now describe the MLT optimization problem.  For convenience, 
we denote the inner product $\tr(AB)$ on the set of symmetric $n\times n$
matrices by $\iprod{A}{B}$. 
\begin{defn}\label{def: opt problem}
Let $G$ be a graph with $n$ vertices.  
Let $D$ be an $n\times (d+1)$ matrix with columns representing
$(d+1)$ samples from an $n$-variate probability distribution.
Let $S = \frac{1}{d}DD^T$ be the sample covariance matrix.
\emph{The MLT optimization problem for $(G,D)$} is to find an $n\times n$ 
positive definite matrix $K$
minimizing $f(K) = \iprod{S}{K} - \log \det K$, 
subject to $K_{ij} = 0$ for all $ij\notin E(G)$.
\end{defn}

The rigidity-theoretic viewpoint requires us to transpose our view of the data matrix. In particular, instead of thinking about $S$ as the sample covariance obtained from $(d+1)$ samples of an $n$-variate distribution, we will think about $S$ as the Gram matrix of a configuration of $n$ points in $(d+1)$-dimensional space.
This allows us to recast the MLT optimization problem in the following equivalent way.

\begin{defn}\label{def: modified opt problem}
Let $G$ be a graph with $n$ vertices and let 
$q$ be a configuration of $n$ vectors in dimension $d+1$.
Let $S = Q^T Q$ be the Gram matrix of $q$.
\emph{The Gram MLT optimization problem for $(G,q)$} is to find an 
$n\times n$ positive definite matrix 
$K$, minimizing $g(K) = \iprod{S}{K} - \log \det K$, subject to 
$K_{ij} = 0$ if $ij\notin E(G)$.
\end{defn}

\begin{lemma}\label{lem: unbounded iff stress}
Let $G$ be a graph with $n$ vertices and let $q$ 
be a configuration of $n$ vectors.
Then the Gram MLT optimization problem (objective function $g$) is unbounded if and only if
there is a nonzero PSD linear equilibrium stress for $q$.
\end{lemma}
\begin{proof}
Let $S$ be the Gram matrix of $q$.  Suppose that 
$\Omega$ is the PSD stress matrix of a non-zero 
linear equilibrium stress for $q$. 
For any $t > 0$, the matrix $I + t\Omega$ is positive definite and
\[
    g(I + t\Omega) = \iprod{S}{I + t\Omega} - \log \det (I + t\Omega).
\]
Since 
\[
    \iprod{S}{I + t\Omega} = \iprod{S}{I} + t\iprod{S}{\Omega}
    = \iprod{S}{I} + t\tr{Q^T Q\Omega} = 
    \iprod{S}{I} + t\tr{Q^T 0} = \iprod{S}{I}
\]
we conclude that 
\[
    g(I + t\Omega) = \tr{S} - \log \det(I + t\Omega) \to -\infty
    \qquad \text{(as $t\to \infty$)}.
\]
So the optimization problem is unbounded.

For the other direction we prove the contrapositive.
Suppose that there is no non-zero 
PSD linear equilibrium stress matrix for 
$q$. We show that the gram MLT optimization problem has a global minimum. Let $\SS$ be be the set of symmetric 
$n\times n$ matrices $\Omega$ with zeros on the non-edges 
of $G$ satisfying $\iprod{\Omega}{\Omega} = 1$. 
For any $\Omega\in \SS$, there is a $t_0 > 0$
so that $K = I + t_0\Omega$ is a feasible point of the 
Gram MLT optimization problem.  Define $t^* \ge t_0 > 0$ to be 
the supremum over values such that $I + t\Omega$ 
is positive definite. 
We will show that, for any 
$\Omega\in \SS$, 
\[
    g(I + t\Omega)\to \infty \qquad \text{(as $t\to t^*$).}
\]
It then follows that, outside of a compact neighborhood of $I$, 
$g(I + t\Omega) > g(I)$, which implies that $g$ has a 
global minimum.
There are two cases.  If $\Omega\in \SS$ is not
PSD, then $t^*$ is finite, and, as $t\to t^*$, 
\[
    g(I + t\Omega) =\tr S + t\iprod{S}{\Omega} - \log \det(I + t\Omega)\to \infty, 
\]
since the last term grows without bound and the linear terms 
have bounded magnitude.  
If $\Omega$ is PSD, then $I + t\Omega$ is positive 
definite for any $t > 0$, and so $t^* = \infty$.
We then have, as $t\to \infty$,
\[
    g(I + t\Omega) = \tr S + t\iprod{S}{\Omega} - 
        \log \det(I + t\Omega) = 
        \tr S + t\iprod{S}{\Omega} - O(\log t)
\]
because the determinant is a polynomial of degree $n$ in
$t$.  Finally, since $S$ and $\Omega$ are 
PSD and $\Omega$ is not a linear equilibrium stress matrix,
$\iprod{S}{\Omega} > 0$, so 
\[
    g(I + t\Omega) \to \infty \qquad 
    \text{(as $t\to \infty$)}.\qedhere
\]
\end{proof}

\subsection{Proof of Theorem \ref{thm: main mlt stress}}
We are now ready to prove Theorem~\ref{thm: main mlt stress}.
Lemmas \ref{lem: main stress forward} and \ref{lem: main stress backwards} below each give one direction. What is left is to rigorously establish the relationship between ``almost all'' and 
generic. The most technical statements are 
handled in Appendix~\ref{appendix: generic}.
Recall that two measures are \emph{mutually absolutely continuous}
if they have the same null sets.

\begin{lemma}\label{lem: main stress forward}
Let $G$ be a graph with $\mlt(G) = d+1$. Then:
\begin{enumerate}[(a)]
    \item there is a generic framework $(G,p)$ in $\mathbb{R}^{d-1}$ with a nonzero PSD equilibrium stress, and
    \item no generic framework $(G,p)$ in $\mathbb{R}^d$ has a nonzero 
    PSD equilibrium stress.
\end{enumerate}
\end{lemma}

\begin{proof}
Let $n$ be the number of vertices of $G$.
Let $D$ be an $n\times d$ data matrix whose columns are i.i.d.~samples from a distribution whose probability measure $\mu$ is mutually absolutely continuous with respect to the Lebesgue measure.
Let $q$ denote the configuration of $n$ points in $\mathbb{R}^d$ given by the rows of $D$.
Since $\mlt(G) = d+1$, the Gram MLT optimization problem for $(G,q)$ is unbounded with positive probability.
Let $X$ denote the set of vector configurations of $n$ points in $\mathbb{R}^d$
for which the Gram MLT optimization problem is unbounded.
Then $X$ is semi-algebraic and not $\mu$-null,
so Lemma~\ref{lem: semi-alg generic} implies that $X$ contains a generic vector configuration, 
which we continue to call $q$.
By Lemma \ref{lem: unbounded iff stress},
there is a non-zero PSD linear equilibrium stress matrix $\Omega$ for $q$.
Since $q$ is generic, it is flattenable.
By Lemma \ref{lem: flattening stress}, the $(d-1)$-dimensional 
framework $(G,p)$ arising from flattening $q$ has an
equilibrium stress matrix with the same signature as $\Omega$,
so this matrix must also be PSD and non-zero.
Finally, Lemma \ref{lem: gen slicing} implies that 
$p$ is generic.
Hence we have constructed a generic $d-1$-dimensional framework $(G,p)$ with a non-zero PSD equilibrium stress.

Let $W$ denote the set of configurations $w$ of $n$ points in $\mathbb{R}^{d+1}$ for which the Gram MLT optimization problem $(G,w)$ is bounded.
Since $\mlt(G) = d+1$, the complement of $W$ is $\mu$-null.
Let $(G,p)$ be a generic framework in $\mathbb{R}^d$.
Scaling the vectors of $\hat p$ by generic weights
gives, via Lemma~\ref{lem: gen slicing}, a generic  configuration $w$ in $\RR^{d+1}$.  
Since $W$ is semi-algebraic and $w$ is generic,
Lemma~\ref{lem: semi-alg generic} implies $w\in W$.
By Lemma~\ref{lem: unbounded iff stress},
there is no non-zero PSD linear equilibrium stress
for $(G,w)$.  By Lemma~\ref{lem: flattening stress},
there is no non-zero PSD equilibrium stress for $(G,p)$.
\end{proof}

\begin{lemma}\label{lem: main stress backwards}
Let $G$ be a graph with $n$ vertices and suppose that $d$
is the smallest dimension so that no generic 
$d$-dimensional framework $(G,p)$ has a non-zero
PSD equilibrium stress.  Then the MLT of $G$ is $d+1$.
\end{lemma}
\begin{proof}
By assumption, there must be
a generic $(d-1)$-dimensional framework 
with a non-zero PSD equilibrium stress, which we will call
$(G,p)$.
By scaling 
the vectors of $\hat p$ by generic numbers $s_i$, 
we obtain, by Lemma \ref{lem: gen slicing}, 
a generic vector configuration $q$ in 
dimension $d$.  By Lemma \ref{lem: flattening stress},
there must be a non-zero PSD linear equilibrium stress for $q$.
Hence, by Lemmas~\ref{lem: unbounded iff stress} and~\ref{lem: semi-alg generic}, the 
set of configurations for which the Gram MLT optimization problem is 
unbounded must be non-null.  We conclude that $\mlt(G) > d$.

Now we take a generic vector configuration $q$ in 
dimension $d+1$.  As noted above, by genericity, $q$
is flattenable, and the flattened $d$-dimensional
point configuration $p$ is also generic by 
Lemma \ref{lem: gen slicing}.  By Lemma \ref{lem: flattening stress}, since $(G,p)$ does not have a non-zero PSD equilibrium 
stress, there is not a non-zero PSD linear equilibrium 
stress for $q$.  Hence, for every generic vector configuration 
$q$ in dimension $d+1$, the Gram MLT optimization problem is bounded by Lemma~\ref{lem: unbounded iff stress}.
Since the set of all such vector configurations is semi-algebraic and contains all the generic points, 
it must have full measure by Lemma \ref{lem: semi-alg generic}.  
This implies that $\mlt(G)\le d+1$.
\end{proof}
The existence of a generic framework in dimension $d-1$
with a non-zero PSD equilibrium stress implies that the 
Gram MLT optimization problem is unbounded with positive probability, 
for any way of sampling data points that is continuous with respect 
to Lebesgue measure.  
However, we don't have a lower bound on
this failure probability.  Any general 
lower bound will be quite bad, since 
Buhl \cite{buhl1993existence} showed that, 
for an $n$ cycle, if the datapoints are sampled 
uniformly from the unit interval, 
the MLE exists after $2$ sample points with 
probability $1 - 2n/n!$, even though the MLT 
of a cycle is $3$.

\subsection{The geometric picture: lifting}
Theorem \ref{thm: main mlt stress}, while precise, and 
as we will see, convenient for deriving bounds on 
the MLT of a graph, is quite technical.  There is 
an underlying geometric idea, that we now 
explain.

\begin{defn}\label{def: liftable}
Let $G$ have $n$ vertices.  Let $(G,p)$ be a 
$d$-dimensional framework.  We say that 
$(G,p)$ is \emph{liftable} if there is an equivalent
$n-1$ dimensional framework $(G,\tilde{p})$
with full affine span.
\end{defn}

The following Lemma is due to Alfakih~\cite{alfakih2011bar}.
For completeness, we provide a proof in Appendix \ref{sec:sep} that uses 
convex geometry ideas from \cite{gt1}.

\begin{lemma}[\cite{alfakih2011bar}]\label{lem: psd lift}
A $d$-dimensional framework $(G,p)$ is liftable if and only if 
it does not have a non-zero PSD equilibrium stress.
\end{lemma}

Theorem \ref{thm: main mlt lift} is immediate 
from Theorem \ref{thm: main mlt stress} and 
Lemma \ref{lem: psd lift}.

\subsection{Remarks}
To close a circle of ideas, we note that much of the
literature on the MLT, including 
\cite{bendavid,blekherman2019maximum,gross2018maximum,uhler2012geometry}
does not work directly with the MLT optimization problem.
Instead, the starting point is the following matrix 
completion problem.
\begin{defn}\label{def: mlt matcomp}
Let $G$ be a graph with $n$ vertices and $S$ an 
$n\times n$ PSD matrix of rank $d+1$.
The \emph{MLT matrix completion problem for $(G,S)$} is to find an $n\times n$ 
positive definite matrix $A$ that has the same diagonal entries as 
$A$ and the same off diagonal entries corresponding to edges of $G$.
\end{defn}
Dempster \cite{dempster1972covariance}
showed that the MLT optmization problem is 
bounded if and only if the MLT matrix completion 
problem is feasible.\footnote{One can also derive Dempster's result via duality
in convex programming, see e.g.~\cite{bendavid}.}
A less direct path to our results is to relate the 
MLT matrix completion for $G$ problem to liftability 
of ``coned'' frameworks $(v_0*G,p)$ \cite{connelly2010global,W83} in one dimension 
higher that have a new vertex $v_0$ connected to 
all the others.  The proof of Theorem \ref{thm: rand rigid 2} in Appendix \ref{sec:app:rand} uses this technique.

Finally, we note that we could have allowed the vector configurations in our optimization problems 
to satisfy a condition strictly weaker than flattenability.
In particular, it would have been enough to only require that the Gram 
matrix $Q^T Q$ have \emph{some} factorization that 
is flattenable, which happens so long as none 
of the vectors in $q$ are zero.  
At the level of frameworks, changing 
factorizations corresponds to projective transformations.
We elected to use the stronger condition to keep the proofs
simpler, and in particular, to avoid
having to define and work with generic low-rank PSD matrices.

\section{MLT bounds from global rigidity}\label{sec: mlt gr}

We can use the results of the previous section along with some facts about global 
rigidity to get improved bounds for the MLT and compute it exactly for some 
interesting families. The main technical tool of this section relates 
generic global rigidity to PSD equilibrium stresses.
\begin{thm}[\cite{cgt1}]\label{thm: gen ur}
Let $G$ be a graph with $n\ge d + 2$ vertices and $d$ a dimension.  
If $G$ globally $d$-rigid, then there is a generic framework $(G,p)$ with a 
PSD equilibrium stress of rank $n - d - 1$.
\end{thm}
We also need a straightforward lemma.
\begin{lemma}\label{lem: mlt monotone}
Let $G$ be a graph and $H$ a subgraph of $G$.  Then $\mlt(H)\le \mlt(G)$.
\end{lemma}
\subsection{Lower bounds}
The main results of this section are new lower 
bounds on the MLT of a graph arising from 
global rigidity in terms of the global rigidity 
number (Def. \ref{def: grn}.)

\begin{proof}[Proof of Theorem \ref{thm: grn mlt}]
Suppose that $G$ is globally $d$-rigid. By Theorem \ref{thm: gen ur}, 
there is a generic framework $(G,p)$ with a non-zero PSD equilibrium stress.  By Theorem \ref{thm: main mlt stress}, $\mlt(G) > d + 1$.
Taking $d$ as large as possible for $G$ to remain   globally $d$-rigid we get 
$\mlt(G) >  \grn(G) + 1$.  The same argument works for any subgraph $H$ of $G$, so Lemma \ref{lem: mlt monotone}
implies that $\mlt(G)\ge \mlt(H) > \grn(H) + 1$.  Maximizing the right-hand side over $H$ we get $\mlt(G) > \grn^*(G) + 1.$
Since $G$ is a subgraph of itself, plainly $\grn^*(G)\ge \grn(G)$.
\end{proof}

We can efficiently compute $\grn(G)$ \cite{gortler2010characterizing}, 
but we do not know the complexity of computing $\grn^*(G)$.  A 
related graph parameter, which may be more computationally tractable is 
the local rigidity analogue.
\begin{defn}\label{def: rigidity number}
Let $G$ be a graph with $n$ vertices.  
The \emph{rigidity number} $\rignum(G)$ is the largest 
$d$ so that $G$ is $d$-rigid and has at least $d+1$ vertices.  
The \emph{subgraph rigidity 
number} $\rignum^*(G)$ is the largest $d$ so 
that $G$ has a subgraph $H$ on at least $d+1$
vertices that is $d$-rigid.
\end{defn}
\begin{thm}\label{thm: rnr mlt}
Let $G$ be a graph.  Then $\rignum(G) + 1\le \rignum^*(G) + 1\le \mlt(G)$.
\end{thm}
The proof needs a result of Jordán.
\begin{lemma}[\cite{J17}]\label{lem: inf rigid d to gr d -1}
Let $G$ be a graph that is $(d+1)$-rigid.  Then $G$ is globally $d$-rigid.
\end{lemma}
\begin{proof}[Proof of Theorem \ref{thm: rnr mlt}]
By Lemma \ref{lem: inf rigid d to gr d -1} one has $\rignum^*(G)\le \grn^*(G)+1$.
Theorem \ref{thm: grn mlt} then implies that 
$\rignum^*(G) + 1 \le \mlt(G)$.
Plainly 
$\rignum(G)\le \rignum^*(G)$, 
giving the last inequality.
\end{proof}
Theorem \ref{thm: rnr mlt} is strictly weaker than 
Theorem \ref{thm: grn mlt}.  For example, for every $n\ge 4$
there are globally rigid graphs in dimension $2$ that have 
$2n - 2$ edges \cite{bergjordan,Bob05}, but if $n>4$ then 
$2n - 2 < 3n - 6$, so these graphs cannot be $3$-rigid. 

The rigidity number of a graph is also easy to compute 
\cite{asimow1978rigidity}.  We do not know the complexity
of computing $\rignum^*(G)$, but, since local rigidity 
is matroidal in nature, tools from submodular optimization 
may apply.

\subsection{Combined bounds and examples}
\label{sec: combined}
Combining what we know so far gives the following.
\begin{thm}\label{rigidityBounds}
For any graph $G$, the following inequalities hold
\begin{enumerate}[(a)]
    \item $\omega(G)\le \rignum^*(G)+1 \le   \grn^*(G)+2 \le \mlt(G)\le\gcr(G)\le \tau(G) + 1$, and
    \item $\rignum(G)+1 \le \grn(G)+2 \le \grn^*(G)+2$.  
\end{enumerate}

\end{thm}

\begin{cor}\label{cor:gr+ind}
If $G$ is both globally $d$-rigid and $(d+1)$-independent, then $\mlt(G) = d+2$.
\end{cor}
We now exhibit two new infinite families of 
graphs $G$ for which the inequalities $\grn(G)+2 \le \grn^*(G)+2 \le \mlt(G) \le \gcr(G)$ are tight.
By applying Lemma \ref{lem: inf rigid d to gr d -1} and Corollary \ref{cor:gr+ind},
we obtain our first infinite family of graphs,
which are the higher dimensional analogue of trees. 

\begin{cor}\label{cor:min rigid implies gcr=mlt}
If $G$ is minimally $d$-rigid, then $\mlt(G) = \gcr(G) = d+1$.
\end{cor}

Our next example is, in essence, an extension of the cycle graphs to higher dimensions.

\begin{defn}\label{defn: circuit}
    $G$ is a \emph{$d$-circuit} if it is not $d$-independent, but every proper subgraph is.
\end{defn}

\begin{cor}\label{cor:gr+cir}
    Let $G$ be a $d$-circuit.  Then $\gcr(G)=d+2$.
    If, furthermore, $G$ is globally $d$-rigid, 
    then $\mlt(G) = d+2$ also.
\end{cor}
\begin{proof}
    Whiteley \cite{W83} proved that $G$ is $d$-independent if and only if the coned graph $v_0*G$, that adds a new vertex $v_0$ connected to 
    every other vertex, is $(d+1)$-independent.
    Since $G$ is a $d$-circuit,
    for any vertex $v$, $G - v$ must be $d$-independent.
    Hence, $v_0*(G - v)$ is $(d+1)$-independent by Whiteley's result.  
    Since $G$ is isomorphic to a subgraph of  $v_0*(G - v)$, 
    it is also $(d+1)$-independent.
    The claim now follows from Corollary \ref{cor:gr+ind}.
\end{proof}

\section{Complete bipartite graphs}\label{sec: complete bipartite}
To test the upper bound from Theorem \ref{thm: uhler gcr mlt} and the 
lower bound from Theorem \ref{thm: buhl}, 
Blekherman and Sinn \cite{blekherman2019maximum}
considered the case of complete bipartite graphs.
They were able to compute the MLT and generic completion ranks exactly, 
obtaining a number of strong results, including 
the first examples of graphs $G$ with $\mlt(G) < \gcr(G)$.

Since, equilibrium stresses of complete bipartite graphs are 
very well understood \cite{bolker1980bipartite,CGbipartite}, 
we have an alternative path to the results from 
\cite{blekherman2019maximum}.
We require the two following results on the 
rigidity theory of bipartite graphs.

\begin{lemma}[\cite{CGTunpublished}]\label{lem: Kmn ggr}
    Fix a $d\in \NN$ and let $m, n\ge d + 1$.  
    If $m + n \ge \binom{d+2}{2} + 1$
    then $K_{m,n}$ is   globally rigid
    in dimension $d$.
\end{lemma}

Given a finite subset $S$ of a vector space, let $D(S)$ denote the linear space of affine dependencies among $S$
and let $S^2$ be the image of $S$ under the \emph{Veronese map} $x\mapsto xx^T$.
The following theorem collects what we need from Bolker and Roth's classic paper~\cite{bolker1980bipartite}.

\begin{thm}\label{thm: bolker roth}
    Let $m,n,d \in \NN$ and let $(K_{m,n},p)$ be a $d$-dimensional framework.
    Let $A,B\subseteq \mathbb{R}^d$ denote the images under $p$ of the partite sets of $K_{m,n}$.
    Then the linear space of equilibrium stresses of $(K_{m,n},p)$ has dimension
    \[
        \dim(D(A))\dim(D(B)) + \dim(D((A\cup B)^2)).
    \]
    Moreover, if $p$ is generic and $m + n \le \binom{d+2}{2}$, then every equilibrium stress matrix has zeros along its diagonal.
\end{thm}
\begin{proof}
    The first claim follows from \cite[Theorem~1]{bolker1980bipartite}.
    The second claim follows from~\cite[Lemma~5]{bolker1980bipartite} and the fact that any set of $\binom{d+2}{2}$ generic symmetric matrices of rank~$1$ is a basis of the space of symmetric $(d+1)\times (d+1)$ matrices.
\end{proof}

\begin{thm}[\cite{blekherman2019maximum}]\label{thm: bipartite 1}
Let $d,m,n\in \NN$ with $m, n\ge d + 2$.
If $m + n \le \binom{d+2}{2}$, then $\mlt(K_{m,n}) \le d + 1$ and $\gcr(K_{m,n}) \ge d + 2$.
\end{thm}
\begin{proof}
Let $(K_{m,n},p)$ be a generic $d$-dimensional framework.
Since $m + n \le \binom{d+2}{2}$, Theorem~\ref{thm: bolker roth} implies that every stress matrix has zeros along its diagonal and is therefore indefinite. Theorem \ref{thm: main mlt stress} 
then implies that $\mlt(K_{m,n}) \le d + 1$.
On the other hand, Theorem~\ref{thm: bolker roth} implies that the space of stresses has dimension at least
$\dim(D(A))\dim(D(B))$, which is positive as $m,n \ge d+2$.
The existence of an equilibrium stress implies that $\gcr(K_{m,n}) \ge d + 2$.
\end{proof}
At this point Theorem \ref{thm: blekherman Kmn} follows quickly.
\begin{proof}[Proof of Theorem~\ref{thm: blekherman Kmn}]
    Theorem \ref{thm: bolker roth} implies that $K_{m,m}$, for $m > 2$,
    is $(m-1)$-independent but not $(m-2)$-independent.  Hence 
    $\gcr(K_{m,m}) = m$.
    By Lemma \ref{lem: Kmn ggr}, for $n > 2$, 
    the global rigidity number of $K_{m,m}$ is the maximum $d$ so that $2m\ge \binom{d+2}{2}+1$.
    For this $d$, Theorem \ref{thm: grn mlt} implies 
    that $\mlt(K_{m,m}) \ge d + 2$.  For any 
    larger $d'$, we have $2m \le \binom{d'+2}{2}$.
    Theorem \ref{thm: bipartite 1} then tells us 
    that $\mlt(K_{m,m})\le (d + 1) + 1 = d + 2$.
    Combining both bounds, we conclude that the MLT of $K_{m,m}$ is the largest 
    $D$ so that $2m > \binom{D+1}{2}$
    as desired.
\end{proof}

\section{A gluing construction}\label{sec: gluing}

In this section we prove some specialized results about giving lower bounds on MLT of graphs. We do this by constructing PSD equilibrium stresses on generic frameworks of a graph obtained by
gluing together smaller frameworks that each have a PSD equilibrium stress.
We will need the following construction from rigidity theory.

\begin{defn}
    Let $G$ be a graph with $n$ vertices and $m$ edges.
    The \emph{rigidity matrix} $R(G,p)$ of a $d$-dimensional framework $(G,p)$ is the $m\times dn$ matrix
    whose rows are indexed by the edges of $G$, columns indexed by the coordinates of $p(1),\dots,p(n)$,
    where the entry corresponding to edge $e$ and $p(v)_i$
    is $p(v)_i-p(u)_i$ if $e = vu$, and $0$ if $v$ is not incident to $e$.
\end{defn}

Given a $d$-dimensional framework $(G,p)$ on a graph $G$ with $n$ vertices, $R(G,p)$ is the Jacobian of the map sending $n$ points in $\mathbb{R}^d$ to the pairwise squared distances corresponding to the edges of $G$, evaluated at $p$.
Equilibrium stresses of $R(G,p)$ are the elements of the left kernel of $R(G,p)$.

\begin{defn}\label{def: clique sum}
A graph $G$ is a \emph{$k$-sum} of two induced 
subgraphs $G_1$ and $G_2$ each with at least $k+1$ vertices 
if $G$ is the union of $G_1$ and $G_2$ and  $G_1\cap G_2$ is isomorphic to $K_k$.
\end{defn}
The following result on equilibrium stresses of frameworks on $k$-sums is standard.
\begin{lemma}\label{lem: d-sum stress}
Let $1\le k\le d+1$ be integers and $G$ a $k$-sum of subgraphs $G_1$ and $G_2$.
Let $(G,p)$
be a $d$-dimensional framework with the vertices of $G_1\cap G_2$ affinely 
independent.
Let $S$ be the space of equilibrium stresses of $(G,p)$ and 
$S_i$ the space of equilibrium stresses of $(G,p)$ supported on the edges of $G_i$.  
Then $S = S_1\oplus S_2$.
\end{lemma}
\begin{proof}
Let $K = G_1\cap G_2$.  
First observe that any equilibrium stress $\omega\in S_1\cap S_2$ 
must be supported by the edges of $K$ and so is an 
equilibrium stress of $(K,p|_K)$.
Since $K$ has at most $d+1$ vertices and is in general affine position, 
$(K,p|_K)$ supports only the zero equilibrium 
stress.  Hence $S_1 + S_2 = S_1\oplus S_2$.

Denote $R_i = R(G_i,p|_{G_i})$.  
The row spans of $R_1$ and $R_2$ are naturally included in the row span 
of $R(G,p)$.  Both of these 
spans include $R(K,p|_K)$.  By general position of the vertices 
corresponding to $K$, this latter space has dimension $\binom{k}{2}$.
So by the interpretation of $S$ and the $S_i$ as cokernels 
of the rigidity matrix and rank-nullity, 
we have
\begin{align*}
    \dim(S) &= m - \rank R(p)\\
    &= m_1 + m_2 - \binom{k}{2} - \rank R(p)\\
    &= m_1 + m_2 - \binom{k}{2} - \rank R_1 - \rank R_2 + \binom{k}{2}\\
    &= m_1 - \rank R_1 + m_2 - \rank R_2\\
    &= \dim(S_1 \oplus S_2)
\end{align*}
and so we can conclude $S = S_1 + S_2 = S_1\oplus S_2$.
\end{proof}

A framework $(G,p)$ is \emph{regular} if its rigidity matrix 
has maximum rank over all frameworks $(G,q)$.
Regularity is preserved 
under non-singular projective transforms applied to $p$.
The converse of the following corollary also true, but we do not need it.

\begin{cor}\label{cor: d-sum rank}
Let $1\le k\le d+1$ and $G$ be a $k$-sum of $G_1$ and $G_2$.  Let $(G,p)$ 
be a $d$-dimensional framework.  If $(G_1,p|_{V(G_1)})$ and 
$(G_2,p|_{V(G_2)})$ are regular then $(G,p)$ is regular.
\end{cor}

\begin{proof}
Let $G_i$ have $n_i$ vertices and $m_i$ edges.  Assume that $p$ is 
such that $(G_1,p|_{V(G_1)})$ and 
$(G_2,p|_{V(G_2)})$ are both regular.  Let $r_i$ be the rank of the 
rigidity matrix of each of these frameworks and $s_i$ the dimension 
of the space of equilibrium stresses.
Observing that $G$ has $n_1 + n_2 - k$
vertices and $m_1 + m_2 - \binom{k}{2}$ edges, we can see that the maximum 
possible rank of the rigidity matrix for $(G,p)$ is $r_1 + r_2 - \binom{k}{2}$.
Since $K = G_1\cap G_2$ is complete and has at most $d+1$ vertices in dimension $d$, 
regularity of $(G_1,p|_{V(G_1)})$ and 
$(G_2,p|_{V(G_2)})$ implies that the vertices of $K$ are affinely independent.  Otherwise there is 
an equilibrium stress supported only on $K$ that is not present in all frameworks.
Hence, we may apply Lemma \ref{lem: d-sum stress} to $(G,p)$ to conclude that its space of 
equilibrium stresses is the direct sum of equilibrium stresses supported on $G_1$ and 
$G_2$ respectively.  The dimension of the space of equilibrium stresses of 
$(G,p)$ is then $s_1 + s_2$.  Then $(G,p)$ is regular since the rank of its rigidity matrix is
\[
    m_1 + m_2 - \binom{k}{2} - s_1 - s_2
    = 
    m_1 + m_2 - \binom{k}{2} - (m_1 - r_1) - (m_2 - r_2)
    = r_1 + r_2 - \binom{k}{2}.\qedhere
\]
\end{proof}
We also have some control of the signs of stress coefficients in PSD 
equilibrium stresses.  The following is from 
 \cite[Lemma 4.9]{cgt2} and the 
discussion around it.
\begin{lemma}\label{lem: psd stress sign}
Let $(G,p)$ be a $d$-dimensional framework and $\omega$ a PSD 
equilibrium stress of $(G,p)$ and $ij$ and edge of $G$ so that 
$\omega_{ij} > 0$.  Then there is a non-singular 
projective transformation $T$ on $\RR^d$ so that $(G,T(p))$
has a PSD equilibrium stress $\psi$ so that $\psi_{ij} < 0$.
\end{lemma}
We have things in place for the main result of this section.
Given a graph $G$ with edge $ij$, we let $G-ij$ denote the graph obtained from $G$ by 
deleting the edge $ij$.

\begin{lemma}\label{lem: psd clique sum}
Let $1\le k\le d$ and $G$ be a $k$-sum of subgraphs $G_1$ and $G_2$
and $ij$ and edge of $G_1\cap G_2$.
Suppose that there are generic $d$-dimensional frameworks
$(G_1,p^1)$ and $(G_2,p^2)$ that, respectively, support 
non-zero PSD equilibrium stresses $\omega^1$ and $\omega^2$,
such that $\omega^k_{ij} \neq 0$ for $k = 1,2$.
Let $G' = G- ij$.
Then there is a generic $d$-dimensional framework $(G',p)$
that supports a non-zero PSD equilibrium stress.
\end{lemma}
\begin{proof}
First assume that $\omega^1_{ij} < 0$ and $\omega^2_{ij} > 0$.  
Since $G_1\cap G_2$ has at most $d$ vertices, any 
affinely independent framework on $G_1\cap G_2$ cannot support 
an equilibrium stress.  Hence, both $\omega^1$ and $\omega^2$ 
have some support outside of $G_1\cap G_2$.
We create a framework $(G,p^0)$ from the frameworks 
$(G_1,p^1)$ and $(G_2,p^2)$ as follows.  Pick a 
non-singular affine map $T$ sending the vertices of 
 $G_1\cap G_2$ in $(G_1,p^1)$ to the corresponding vertices 
in $(G_2,p^2)$ and apply it to $p^1$.  The defines a 
framework $(G,p^0)$.

By Corollary \ref{cor: d-sum rank} and the  
genericity of $(G_i,p^i)$, the framework $(G,p^0)$ 
is regular.  
Since equilibrium stresses are preserved under affine maps, 
$\omega^1$ and $\omega^2$ are both equilibrium stresses of $(G,p^0)$. 
Our assumptions about the 
signs imply that some positive linear combination $\omega$ of 
$\omega^1$ and $\omega^2$
has vanishing coefficient on the edge $ij$.  Because the $\omega^i$ 
have some necessarily disjoint support, $\omega$ is non-zero.
Since a positive 
combination of PSD equilibrium stresses is PSD, we conclude that $\omega$ is.  
Since $\omega$ is not supported on $ij$, it 
is also an equilibrium stress of $(G',p^0)$.
Potentially, $(G',p^0)$ is not generic, but since it is regular, a  
small perturbation $(G,p)$ 
that is generic will have an equilibrium stress 
close to $\omega$ that is also PSD.

If $\omega^1_{ij} > 0$, we reduce to the previous case by 
applying a projective transformation, as in Lemma \ref{lem: psd stress sign}.
The argument is then the same as before, since we only used that the $(G_i,p^i)$
are generic to make them regular.  Regularity is preserved by 
projective transformations.
\end{proof}

\subsection{Remarks}
A natural question is whether the lower bound in Theorem \ref{thm: grn mlt}
is tight.  The results of this section show that it is not.  By Lemma 
\ref{lem: psd clique sum}, if we let $G$ be the $2$-sum of two copies
of $K_{d+2}$ over an edge $ij$, and $G'$ the graph $G-ij$,
there is a generic framework $(G',p)$ in dimension $d$ with a non-zero 
PSD equilibrium stress.  Theorem \ref{thm: main mlt stress}, 
then implies that $\mlt(G')\ge d + 2$.  On the other hand,
since every induced subgraph of $G'$ is independent in dimension 
$d$, $\grn^*(G')\le d-1$.  Hence, $\grn^*(G') + 2 < \mlt(G')$.

If we ask, in addition that $G$ is $(d+1)$-connected, we do 
not know an example where the lower bound in Theorem 
\ref{thm: grn mlt} is not tight.  

\section{Equality of small MLT and GCR}\label{sec: mlt 3}
In this section, we prove Theorem \ref{thm: equality of gcr and mlt}, which rests on the 
rich combinatorial theory of  2-rigidity of graphs (see e.g.~\cite{lee-streinu} for an overview).
The cornerstone of this theory is Theorem~\ref{thm: LPG}, the Laman--Pollaczek-Geiringer theorem.
We begin with the necessary definitions.

\begin{defn}\label{def: laman}
A graph $G$ with $n$ vertices is $(2,3)$-sparse if, for all subgraphs 
with $n'$ vertices and $m'>0$ edges, $m' \le 2n' - 3$.
If $G$ is $(2,3)$-sparse and, in addition has $2n - 3$ edges, it is called a 
\emph{Laman graph}.
A graph that is not $(2,3)$-sparse, but becomes so after removing any 
edge is called a \emph{Laman circuit}.
\end{defn}

\begin{thm}[\cite{LAMAN,HPG}]\label{thm: LPG}
A graph $G$ is $2$-independent if and only if $G$ is $(2,3)$-sparse.
\end{thm}

Via Theorem~\ref{thm: gross sullivant}, Theorem~\ref{thm: LPG} immediately gives us a 
combinatorial characterization of the graphs with $\gcr(G) = 3$; these are the $(2,3)$-sparse 
graphs that contain a cycle.
As we will see in Proposition~\ref{prop: mlt 3 laman}, this also characterizes graphs 
with $\mlt(G) = 3$.
In order to prove this, we need the following lemma which makes crucial use of Berg and 
Jordan's~\cite{bergjordan}
combinatorial characterization of global rigidity in two dimensions.

\begin{lemma}\label{lem: 2d PSD circuit stress}
Let $G$ be a Laman circuit.  Then there are generic $2$-dimensional
frameworks $(G,p)$ satisfying a non-zero PSD equilibrium stress.
\end{lemma}

\begin{proof}
If $G$ is $3$-connected, a result of Berg and Jordán \cite{bergjordan}  implies that $G$ is 
  globally rigid.  The desired statement then follows from Theorem \ref{thm: gen ur}.
If $G$ is not $3$-connected, we can find a $2$-separation $\{x,y\}\subseteq V(G)$ in $G$.
A counting argument \cite[Lemma 2.4, inter alia]{bergjordan} 
implies that
$xy$ is not an edge of $G$ and that $G\cup \{xy\}$ 
is a $2$-sum of smaller 
Laman circuits $G_1$ and $G_2$.  
By induction, we may assume that there are generic $2$-dimensional frameworks 
$(G_1,p^1)$ and $(G_2,p^2)$ that each support a PSD equilibrium stress $\omega^1$ and 
$\omega^2$. Since $G_1$ and $G_2$ are circuits, the supports of $\omega^1$ 
and $\omega^2$ include the edge $xy$.  By Lemma \ref{lem: psd clique sum}, there is then a 
generic framework $(G,p)$ with a non-zero 
PSD equilibrium stress.
\end{proof}

\begin{prop}\label{prop: mlt 3 laman}
Given a graph $G$, the following are equivalent:
\begin{enumerate}[(a)]
    \item\label{item: combinatorial} $G$ is $(2,3)$-sparse and contains a cycle,
    \item\label{item: gcr 3} $\gcr(G) = 3$, and
    \item\label{item: mlt 3} $\mlt(G) = 3$.
\end{enumerate}
\end{prop}

\begin{proof}
Theorems~\ref{thm: LPG} and~\ref{thm: gross sullivant} imply that $\gcr(G) = 3$ if and only if
$G$ is $(2,3)$-sparse and contains a cycle.
Now assume $\gcr(G) = 3$.
Since cycles are globally $1$-rigid, 
any graph $G$ with a cycle has $\grn^*(G)\ge 1$, so 
$\mlt(G)\ge 3$ by Theorem \ref{thm: grn mlt}.
On other other hand, if a graph $G$ is  
$(2,3)$-sparse then $\gcr(G)\le 3$ and so $\mlt(G)\le 3$ follows from Theorem~\ref{thm: uhler gcr mlt}.

If $\gcr(G) \le 2$ or $\mlt(G) \le 2$, then $G$ cannot have a cycle.
So assume $\gcr(G) \ge 4$.
Theorems~\ref{thm: LPG} and~\ref{thm: gross sullivant} now imply that $G$ contains a Laman circuit $H$ as a subgraph.
By Lemma \ref{lem: 2d PSD circuit stress},
$H$ has a generic $2$-dimensional framework $(H,p)$ with non-zero PSD equilibrium stress.
Theorem \ref{thm: main mlt stress} 
implies $\mlt(H) \ge 4$ and therefore Lemma~\ref{lem: mlt monotone} implies $\mlt(G) \ge 4$.
\end{proof}

We are now ready to prove the main result of this section.

\begin{proof}[Proof of Theorem~\ref{thm: equality of gcr and mlt}]
    As noted in~\cite{gross2018maximum}, $\mlt(G) = 1$ if and only if $G$ has no edges and $\mlt(G) = 2$ if and only if $G$ has no cycles.
    In both cases, it is easy to see that $\gcr(G) = \mlt(G)$.
    If $\mlt(G) = 3$ or $\gcr(G) = 3$, then $\mlt(G) = \gcr(G)$ follows from Proposition~\ref{prop: mlt 3 laman}.
    If $\gcr(G) = 4$, then
    Theorem~\ref{thm: LPG} and Lemma~\ref{lem: 2d PSD circuit stress} together imply that $\mlt(G) \ge 4$
    and equality follows from Theorem~\ref{thm: uhler gcr mlt}.
\end{proof}

\begin{rmk}\label{rmk: best possible}
Theorem~\ref{thm: equality of gcr and mlt} is best possible in the sense that if $a \ge 4$ and $b \ge 5$,
then there exist graphs $G,H$ such that $\mlt(G) = a < \gcr(G)$ and $\mlt(H) < b = \gcr(H)$.
In particular, let $n = \left\lfloor\frac{1}{2}\binom{a+1}{2}\right\rfloor$,
and let $D$ be the smallest $k$ such that $\binom{k+1}{2} \ge 2b$.
Then, Theorem~\ref{thm: blekherman Kmn} implies $\mlt(K_{n,n}) = a < n = \gcr(K_{n,n})$ and that $\gcr(K_{b,b}) = b > D =\mlt(K_{b,b})$.
\end{rmk}

\section{Weak maximum likelihood threshold}\label{sec: weak MLT}
This section includes connections between the \emph{weak} maximum likelihood threshold of a graph, and two areas of classical combinatorics: partially ordered sets, and graph dimension (i.e.~the minimum dimension in which a graph can be realized as a unit-distance graph).

\begin{defn}\label{def: wmlt}
The \emph{weak maximum likelihood threshold} of a graph $G$, denoted 
$\wmlt(G)$ is the 
smallest number of samples\footnote{Again, we are assuming that the samples are i.i.d.~from a distribution whose probability measure is mutually absolutely continuous with respect to Lebesgue measure.} required for 
the MLE of the Gaussian graphical model associated with $G$ to exist
with positive probability.
\end{defn}
The definition of $\wmlt(G)$ is the same as that of $\mlt(G)$, but with the phrase ``almost surely'' swapped out for ``with positive probability.'' 
Arguments along the lines of Section \ref{sec: lifting} then yield
the analogue of Theorem \ref{thm: main mlt lift}.
Since the proof is very similar, we skip it.

\begin{prop}\label{prop:wmlt liftable}
Let $G$ be a graph with $n$ vertices.  The WMLT of $G$ is $d+1$
if and only if $d$ is the smallest 
dimension such that some generic $d$-dimensional framework $(G,p)$
is liftable.  
\end{prop}

The following implies that we can ignore genericity of our witness (cf.~\cite[Definition~5.1]{gross2018maximum}).
\begin{prop}
\label{prop: wmlt matrix completion}
Let $d\in \NN$ be a dimension and $G$ be a graph with $n\ge d+1$ vertices.  
If there is any liftable $d$-dimensional framework $(G,p)$ then there 
is a generic liftable $d$-dimensional framework.  In particular, 
$\wmlt(G) \le d+1$.
\end{prop}
\begin{proof}
Let $(G,p)$ be a liftable $d$-dimensional framework.  By Lemma 
\ref{lem: psd lift}, $(G,p)$ does not have a non-zero PSD equilibrium 
stress.  By lower semi-coniuity of the rank of the rigidity 
matrix, there is a nbd $U$ of $p$ so that if $q\in U$, the space of 
equilibrium stresses of $(G,q)$ has dimension at most that of $(G,p)$.
Hence any equilibrium stress of $(G,q)$ is a small perturbation of a 
stress of $(G,p)$.  For sufficiently small perturbations, signature is
preserved, so some neighborhood of $p$ consists of only frameworks
without non-zero PSD equilibrium stresses.  This neighborhood 
contains a generic framework.
The second statement follows from Proposition \ref{prop:wmlt liftable}.
\end{proof}

\subsection{Existing bounds on the WMLT}
The weak maximum likelihood threshold of a graph is one 
if and only if it has no edges. Examples of graphs for which MLT = WMLT = $d+1$ are the \emph{$d$-laterations}; i.e.,
graphs formed from $K_{d+1}$ by a sequence of $d$-dimensional 0-extensions.
Other than this, very little is known.
Gross and Sullivant~\cite{gross2018maximum} showed that $\wmlt(G)$ is at most the chromatic 
number of $G$.
Buhl~\cite{buhl1993existence} characterized the weak maximum likelihood thresholds of cycles, 
showing that $\wmlt(G) = 3$ if $G$ is a three-cycle, and $\wmlt(G) = 2$ when $G$ is a cycle 
of length four or greater.
\begin{defn}
    Given a directed graph, a cycle in the underlying undirected graph is \emph{stretched} 
    if it is of the form $v_1\rightarrow v_2 \rightarrow \dots \rightarrow v_k \leftarrow v_1$.
    Given $(G,p)$ is a framework in $\mathbb{R}^1$ with no edges of length zero, 
    a cycle in $G$ is \emph{stretched} if the corresponding cycle is stretched in the 
    orientation of $G$ obtained by directing each edge $i$ 
    $j$ towards $j$ if $p(j) > p(i)$ and otherwise toward $i$.
\end{defn}
The following proposition can be seen as the rigidity-theoretic version 
of~\cite[Theorem 4.3]{buhl1993existence}. 
\begin{prop}\label{prop:stretchedCycleStress}
Let $G$ be a cycle and let $(G,p)$ be a generic framework in $\mathbb{R}^1$.
Then $(G,p)$ has a non-zero PSD equilibrium stress if and only if it is a stretched cycle.
\end{prop}
The proposition is a special case of a more general 
statement due Kapovich and Millson \cite{kapovich}
which we discuss in Appendix~\ref{appendix: cycle stress signature}.

Proposition~\ref{prop:stretchedCycleStress} immediately 
implies the following result of Gross and Sullivant.
\begin{cor}[{\cite[Corollary 5.4]{gross2018maximum}}]\label{cor: no stretched cycles}
 If $\wmlt(G) = 2$, then $G$ has an acyclic orientation with no stretched cycles.
\end{cor}
In \cite{gross2018maximum}, the property of having an acyclic orientation with 
no stretched cycles is called \emph{Buhl's cycle condition}.

\subsection{A conjecture and a connection}
Based on experimental evidence, we believe that the converse to 
Corollary~\ref{cor: no stretched cycles} is true.
\begin{conj}\label{conj: wmlt 2}
    If $G$ has at least one edge and an acyclic orientation with no stretched cycles, then $\wmlt(G) = 2$.
\end{conj}
Directed acyclic graphs with no stretched cycles are well-studied objects in combinatorics: 
they are  diagrams of partially ordered sets. It is NP-hard to determine whether a 
given undirected graph has an acyclic orientation with no stretched 
cycles~\cite{brightwell1993complexity}. Thus Conjecture~\ref{conj: wmlt 2} would imply that 
the decision problem of whether a given graph has $\wmlt(G) = 2$ is NP-hard. 
Via the coning construction \cite{W83}, 
this would imply that determining weak MLT is NP-hard in general.

The following definition is due to Erd\"os, Harary, and Tutte.
\begin{defn}[\cite{erdos1965dimension}]
The \emph{dimension} of a graph $G$, denoted $\dim(G)$, is the minimum 
$d$ such that there exists a framework $(G,p)$ in $\mathbb{R}^d$ such 
that $\|p(i)-p(j)\| = 1$ for all edges $ij \in E(G)$.
\end{defn}
The \emph{Hadwiger-Nelson problem} is a longstanding open problem in
combinatorics which asks for the maximum chromatic number of a graph 
$G$ with $\dim(G) = 2$. See~\cite{de2018chromatic} for the most 
recent progress and a brief account of the history. The connection 
to weak maximum likelihood thresholds is given by the following.
\begin{prop}
    Let $G$ be a graph.
    Then $\wmlt(G) \le \dim(G)+1$.
\end{prop}
\begin{proof}
Let $(G,p)$ be a framework in $\mathbb{R}^{\dim(G)}$ so that every 
edge of $G$ has length $1$. Then $(G,p)$ is liftable. 
A suitable witness is the framework $(G,q)$ in 
$\mathbb{R}^{n-1}$ where the $q(i)$s are the vertices of a 
suitably scaled unit simplex. The result now follows 
from Proposition~\ref{prop:wmlt liftable}.
\end{proof}
It is well-known that $\dim(G) + 1 \le \chi(G)$.  Indeed, if $G$ 
has chromatic number $d+1$, then there is a unit-distance embedding 
of $G$ in dimension $d$ by putting each of the $d+1$ 
color classes on a distinct vertex of a regular simplex in 
dimension $d$.  Hence, this result improves the 
inequality $\wmlt(G) \le \chi(G)$ from \cite{gross2018maximum}.

\begin{appendix}

\section{``almost all'' vs ``generic''}\label{appendix: generic}
In this appendix, we prove some technical 
results needed for Theorem \ref{thm: main mlt stress}.
We begin with a precise definition of genericity.

\begin{defn}
    A point $x \in \mathbb{R}^n$ is \emph{generic} if its coordinates are algebraically independent over $\mathbb{Q}$.
    If $S \subseteq \mathbb{R}^n$ is an irreducible semi-algebraic set, then a point $x \in S$ is \emph{generic in $S$}
    if whenever a polynomial $f$ with rational coefficients satisfies $f(x) = 0$, then $f(y) = 0$ for all $y \in S$.
\end{defn}

We record some facts about semi-algebraic sets (see,
e.g.~\cite{bochnak,stasica}). Recall that a \emph{finite boolean combination} of sets $\{S_\alpha\}_{\alpha \in J}$
is a set obtained using finitely many unions and intersections of sets in $\{S_\alpha\}_{\alpha \in J}$.

\begin{lemma}\label{lem: semi-alg facts}
Let $S$ be an irreducible semi-algebraic set and $X\subseteq S$
semi-algebraic. Then:
\begin{enumerate}[(a)]
    \item $X$ is a finite boolean combination of open and closed (standard topology) subsets of $S$,
    \item $X$ contains an open subset of $S$ if and only if it has the same dimension as $S$,
    \item if $X$ is of lower dimension than $S$, then each $x \in X$ satisfies some polynomial with coefficients in the field of definition for $X$ that is not satisfied by some points in $S$, and
    \item $X$ has finitely many irreducible components.
\end{enumerate}
\end{lemma}

\begin{lemma}\label{lem: semi-alg measure}
Let $S\subseteq \RR^N$ be an irreducible semi-algebraic 
set, $X\subseteq S$ semi-algebraic, and suppose that $\mu$
is a Borel measure on $S$ that is mutually absolutely 
continuous with respect to Lebesgue measure on $S$.
Then $X$ is $\mu$-null if and only if every irreducible component
of $X$ is of lower dimension than $S$.
\end{lemma}

\begin{proof}
From Lemma \ref{lem: semi-alg facts} and the fact that $\mu$
is a Borel measure, we know that each irreducible component $Y$
of $X$ is measurable.  Since $\mu$ is a  
Borel measure and Lebesgue measure on $S$ is absolutely 
continuous with respect to $\mu$, if $Y$ contains an open subset of 
$S$, $\mu(Y) > 0$.  Hence, if $Y$ has the same dimension as
$S$, we must have $\mu(Y) > 0$.  On the other hand, if 
$Y$ is of lower dimension then the 
(standard topology)
closure 
$\overline{Y}$ is closed and nowhere dense.  Absolute continuity of $\mu$
with respect to Lebesgue measure
then implies that $\mu(Y) \le \mu(\overline{Y}) = 0$.
Repeating this argument for each irreducible component
of $X$ completes the proof.
\end{proof}

To translate between generic statements 
and measure theoretic ones, we use the following.

\begin{lemma}\label{lem: semi-alg generic}
Let $S$ be an irreducible semi-algebraic subset of $\RR^N$ and let $X$ be 
a semi-algebraic subset of $S$, with both $S$ and $X$ defined over $\QQ$.
Let $\mu$ be a Borel measure on $S$ mutually absolutely
continuous with respect to Lebesgue measure.
Then:
\begin{enumerate}[(a)]
    \item if $X$ is $\mu$-null, then no generic points of $S$ are in $X$,
    \item if $X$ has full $\mu$-measure, then every generic point of $S$ is in $X$, and
    \item if neither $X$ nor its complement are $\mu$-null, then some generic points of $S$ are in $X$ and some are not.
\end{enumerate}
\end{lemma}
\begin{proof}
Suppose, for the moment, that $X$ is 
irreducible.  By Lemma \ref{lem: semi-alg measure} if $X$ is $\mu$-null 
it is of lower dimension than $S$.  By Lemma \ref{lem: semi-alg facts},
no point of $X$ can be generic.
In general, we repeat the argument for each irreducible component, 
which gives (a).  Part (b) follows from (a) via complementation.

For (c), Lemma \ref{lem: semi-alg measure} implies that a $\mu$-non-null
semi-algebraic set contains an open set.  Any non-generic point 
must lie in a nowhere dense algebraic subset of $S$, so if both $X$
and its complement are $\mu$-non-null both contain a generic point.
\end{proof}

\begin{lemma}\label{lem: gen slicing}
Let $v$ be a generic configuration of $n$ vectors in
$\RR^{d+1}$.  Then $v$ is flattenable, and the flattened
configuration $p$ in $\RR^d$ is also generic.
Conversely, if $p$ is a generic configuration of $n$ 
points in $\RR^d$, then there is a generic vector configuration 
$v$ in $\RR^{d+1}$ so that $p$ is the flattening of $v$.
\end{lemma}

\begin{proof}
First suppose that $v$ is a generic configuration of $n$
vectors in $\RR^{d+1}$.  Letting $t_i$ be the last coordinate 
of $v(i)$, we notice that if $t_i = 0$ for any $i$, 
then $v$ satisfies a non-trivial polynomial equation 
and so is non-generic.  Hence, $v$ is flattenable.  
The map sending a flattenable vector configuration 
$v$ to its flattening $p$ is rational and surjective
onto configurations of $n$ points in $\RR^d$.
The result now follows from \cite[Lemmas 2.7 and 2.8]{gortler2010characterizing}.
\end{proof}

\section{Equilibrium stresses and convexity}\label{sec:sep}

The goal of this appendix is to give a self-contained proof of Lemma~\ref{lem: psd lift},
which originally appeared in~\cite{alfakih2011bar}.
We will denote the interior of a set $S$ by $\intr(S)$.

\begin{lemma}
\label{lem:convProj}
Let $K$ be a convex $n$-dimensional set in $\RR^n$,
let $\pi$ be a rank-$m$ linear projection from $\RR^n$ to $\RR^m$, and
let $k:=\pi(K)$ be the $m$-dimensional image.
The following are equivalent:
\begin{enumerate}[(a)]
    \item 
$\pi^{-1}(x) \cap \intr(K)$ is nonempty,
\item $x  \in \intr(k)$, and
\item  $x$ does not lie on a supporting hyperplane for $k$.
\end{enumerate}
\end{lemma}
\begin{proof}
Equivalence of the latter two conditions follows from the supporting hyperplane theorem \cite[Ch. 8]{simon-convexity}.

If $\pi^{-1}(x) \cap \intr(K)$ is nonempty, then there is a point $X \in \pi^{-1}(x)$ with open neighborhood $N$ satisfying
$N \subseteq \intr(K)$. Since $\pi$ is a linear 
map, it is open onto its image, which is $\RR^m$,
so $\pi(N)$ is open.  Since $x\in \pi(N) \subseteq k$, 
it is interior in $k$ (here we used that $K$ is full-dimensional, 
so that $\intr(K)$ is open and nonempty in $\RR^n$).

Now assume $\pi^{-1}(x) \cap \intr(K) = \emptyset$.
Since $\pi^{-1}(x)$ is convex,  
there must be an affine
hyperplane $H$ in $\RR^n$
weakly separating $\pi^{-1}(x)$ from $\intr(K)$.  
Since $\pi^{-1}(x)$ is an affine subspace, we have 
$\pi^{-1}(x) \subseteq H$.
Let $\ell$ be the linear functional and $\alpha$ the real number so that
\[
    H = \{y\in \RR^N : \ell(y) = \alpha\}.
\]
Since $\pi^{-1}(x)$ is parallel to the kernel of 
$\pi$, we have that $\ker \ell \supseteq \ker \pi$.
Hence, we have a well-defined linear map 
$\overline{\ell} : \RR^m\cong \RR^n/\ker \pi \to \RR$ given 
by 
\[
    \overline{\ell}(x') = \ell(y') \qquad \text{(any $y'\in \pi^{-1}(x')$)}.
\]
Hence, since $H$ supports $K$, the hyperplane $\{y \in \mathbb{R}^m: \overline{\ell}(y) = \alpha\}$ supports $k$ at $x$.
\end{proof}

\begin{proof}[Proof of Lemma~\ref{lem: psd lift}]
Let $K$ be the PSD cone.
Points in $K$ are Gram matrices of
$n$-point configurations in $\RR^n$; points in 
the interior correspond to configurations with 
$n$-dimensional linear span. Such configurations 
will have an $n-1$-dimensional affine span.
Fixing a graph $G$ with $m$ edges,
$\pi$ will be the map to $\RR^m$, which measures the 
squared lengths of the corresponding framework; i.e.~indexing $\mathbb{R}^m$ by the edges of $G$,
for each edge $ij$ of $G$, we have
\[
    \pi(X)_{ij} = X_{ii} + X_{jj} - 2X_{ij}.
\]
The image $\pi(K)$ is an $m$-dimensional convex cone $k \subseteq \mathbb{R}^{m}$.
Using Lemma~\ref{lem:convProj}, it now suffices to show that $\pi(p)$ lies on the boundary of $k$ if and only if $(G,p)$ has a PSD equilibrium stress.

Given a configuration $q$ of $n$ points in $\mathbb{R}^n$,
let $q^{l}$ denote the vector $q(1)_l,\dots,q(n)_l$ consisting of the $l^{\rm th}$ coordinate of each point in $q$.
Now assume $\pi(p)$ lies on the boundary of $k$,
let $\omega$ be the normal vector of the hyperplane tangent to $k$ at $\pi(p)$, and let $\Omega$ be the matrix obtained by setting $\Omega_{ij} = \Omega_{ji} = -\omega_{ij}$ for all edges $ij$ of $G$, $\Omega_{ii} = \sum_{j} \omega_{ij}$ for $i = 1,\dots, n$, and all other entries zero.
This means that for any configuration $q$ of $n$ points in $\mathbb{R}^n$,
the following inequality holds, and is moreover an equality when $q = p$:
\begin{align}\label{eq: supporting hyperplane configuration}
0 \le  
 \sum_{ij {\rm \ edge \ of \ } G} \omega_{ij} ||(q_i-q_j)||^2 =
\sum_{l} (q^l)^T\Omega q^l.
\end{align}
This implies that $\Omega$ is PSD and that $(p^l)^T\Omega p^l = 0$ for all $l$.
Together, these imply that $\Omega p^l = 0$ for each $l$, which is exactly the condition for $\omega$ to be an equilibrium stress of $p$.
Thus $\omega$ is a PSD equilibrium stress for $p$.

Finally, note that if $\Omega$ is a PSD equilibrium stress of $p$, then the above arguments can be reversed to show that $\Omega$ defines a supporting hyperplane of $k$ at $p$.
\end{proof}

\section{the signature of a cycle stress}\label{appendix: cycle stress signature}
\begin{defn}
    Let $C_n$ denote the directed cycle on vertex set $\{1,\dots,n\}$ with
    edges $1\rightarrow 2, 2\rightarrow 3,\dots,(n-1)\rightarrow n, n\rightarrow 1$.
    A \emph{framework} $(C_n,p)$ on $C_n$ refers to a framework on the undirected graph underlying $C_n$.
    In a general position framework $(C_n,p)$ in $\mathbb{R}^1$, an edge $i \rightarrow (i+1)$ is \emph{forwards} if $p(i) < p(i+1)$ and \emph{backwards} otherwise.
\end{defn}

Note that every general position framework $(C_n,p)$ in $\mathbb{R}^1$ has at least one forwards edge and at least one backwards edge.
Proposition~\ref{prop:stretchedCycleStress} is a corollary of the following, which classifies the signatures of the stresses of cycles in $\RR^1$.

\begin{thm}[\cite{kapovich}]\label{thm: cycle stress}
Let $(C_n,p)$ be a generic framework in $\mathbb{R}^1$ and let $f$ be the number of forwards edges and $b$ the number of backwards edges.
Then, for every nonzero equilibrium stress matrix $\Omega$ of $(C_n,p)$, the signature of $\Omega$ is either $(f-1,2,b-1)$ or $(b-1,2,f-1)$.
\end{thm}

It is easy to see from Theorem~\ref{thm: cycle stress} that any cycle framework 
$(C_n,p)$ with exactly one backwards edge, or exactly one forwards edge, must be stretched and vice versa. Since the proof in \cite{kapovich} uses Hodge theory, we provide a linear-algebraic argument.

\begin{lemma}\label{lem: cycle stress signs}
Let $(G,p)$ be a general position framework in $\RR^1$ with the edge $\{1,n\}$, 
wlog (after cyclic relabeling) backwards and $p(n)$ the rightmost vertex.
Then $(G,p)$ has a unique, up to nonzero scaling, equilibrium stress $\omega$, and this scaling can be chosen such that every forward edge has positive coefficient and every negative edge has 
negative coefficient.
\end{lemma}
\begin{proof}
Without loss of generality, set the coefficient $\omega_{1,n}$ on the edge $\{1,n\}$ to 
$-1$.  Now walk, in cyclic order, starting from vertex $1$, 
solving the equilibrium condition locally, by setting 
$\omega_{i,1+1}$ to solve (indices taken cyclically):
\[
    \omega_{i,i+1}(p(i+1) - p(i)) = \omega_{i-1,i}(p(i) - p(i-1)).
\]
Notice that the sign changes whenever we switch from forwards to 
backwards edges, so we have the desired sign pattern.  General 
position implies that we do not get any zero coefficients.
We have, automatically, equilibrium at every vertex except, 
possibly $p(n)$.  To check that we have equilibrium, notice that, 
by induction, all the vectors $\omega_{i,i+1}(p(i+1) - p(i))$
have magnitude $|p(n) - p(1)|$ and that $\omega_{n-1,n}$ is 
positive.
\end{proof}
Our next lemma is a general fact that can be verified by 
direct computation.
\begin{lemma}\label{lem: stress energy cut}
Let $H$ be any graph and $\omega$ an equilibrium stress with associated matrix $\Omega$.
For any subset $S$ of vertices of $H$, let $x(S)$ be the characteristic vector of $S$.
Then 
\[
    x(S)^T\Omega x(S) = \sum_{\text{edges $ij$: $i\in S$, $j\notin S$}} \omega_{ij}.
\]
In particular, if $S$ is the set of 
vertices on one side of a cut consisting of edges with positive (resp negative) stress coefficients, then $x(S)$ has positive (resp negative) Rayleigh quotient.
\end{lemma}

\begin{proof}[Proof of Theorem \ref{thm: cycle stress}]
Let $(G,p)$ be as in the statement and $\Omega$ 
scaled as in Lemma \ref{lem: cycle stress signs}.  Uniqueness, up to nonzero scale,
of the equilibrium stress on $(G,p)$ implies this is 
possible.
Now recall that removing any two edges from a cycle determines a 
cut.  If we have $b$ backwards edges, $e_1, \ldots, e_b$, 
each of the cuts $\{e_1, e_j\}$ for $2\le j\le b$ gives 
rise to a collection of $b-1$ independent incidence vectors with negative
Rayleigh quotient, from Lemma \ref{lem: stress energy cut}.  
Hence $\Omega$ has at least $b-1$ negative eigenvalues.
Similarly, the $f$ edges $e'_1, \ldots, e'_f$ with positive stress
coefficients give $f-1$ independent incidence vectors with positive 
Rayleigh quotient.
Since $\Omega$ has a nullity of at least $2$, as an equilibrium 
stress matrix, the proof is complete.
\end{proof}

\section{Random graphs}
\label{sec:app:rand}
In this section, we prove 
Theorem \ref{thm: rand rigid 2}.  
We will use results of Candès and Tao \cite{candes-tao} as a 
``black box'' along with some ideas from Saliola and 
Whiteley \cite{whiteley-pseudo}. 

The results in this appendix are based on those from 
\cite{KT-random}, but slightly weaker.  We include the 
proofs here to keep this paper self-contained and because 
they are somewhat simpler. 

\subsection{Low-rank matrix completion}
Matrix completion is the problem of imputing an 
unknown $n\times n$ matrix $M$ from a 
subset of its entries.  A fundamental 
result of Candès and Tao 
(specialized to the case where $r$ is a
fixed constant is):
\begin{thm}[{\cite{candes-tao}}]\label{thm: ct-matcomp}
If $M$ is an 
$n\times n$ symmetric rank $r$ matrix 
with the strong incoherence property with constant $\mu > 0$,
there is a constant $C_{r,\mu} > 0$ such that, if 
\[
    m\ge C_r \mu^2 n(\log n)^2
\]
entries of $M$ are sampled uniformly at random, then 
with probability at least $1-n^{-3}$, a nuclear 
norm minimization algorithm recovers all of $M$
from the observed entries.
\end{thm}
For our purposes we do not need to know the exact 
definition of strong incoherence.  What we 
do need is that Candès and Tao also 
prove:
\begin{thm}[{\cite[Sec. 1.5.1]{candes-tao}}]\label{thm: ct-incoherent}
For any rank $r$ and some $\mu = O(\sqrt{\log n})$, there are open sets of 
matrices satisfying the strong incoherence property.
\end{thm}
In particular, there is a generic symmetric rank $r$ 
matrix satisfying the strong incoherence property.
\begin{thm}\label{thm: gen matcomp random}
Let $r$ be a fixed rank.  
There is a constant $C'_r > 0$ so that, if
\[
  m\ge C'_r n(\log n)^3  
\]
entries of a generic symmetric matrix $M$ of rank $r$ 
are sampled uniformly at random, there are finitely many 
symmetric matrices $M'$ of rank $r$ agreeing with $M$
on the observed entries with probability 
at least $1-n^{-3}$.
\end{thm}
For a set of observed entries, if, for a generic rank 
$r$ symmetric matrix $M$, there are finitely many 
rank $r$ matrices agreeing with $M$ on the observed entries, 
the pattern is said to be \emph{finitely completable}.
\begin{proof}
Let $M$ be a generic, rank $r$ symmetric matrix with the strong 
incoherence property where $\mu = O(\sqrt{\log n})$ from 
Theorem \ref{thm: ct-incoherent}.  Take $C'_r$ to be $C_r$ from 
the statement of Theorem \ref{thm: ct-matcomp}.
Such an $M$ exists by Theorem \ref{thm: ct-incoherent}.
By Theorem \ref{thm: ct-matcomp}, if $m$ (from the statement) 
entries are sampled uniformly at random, the nuclear norm
minimization algorithm finds $M$ with probability $1 - n^{-3}$.

If there were any other way to complete the observed entries 
to get a different rank $r$ matrix, this would be impossible
as soon as the success probability rises above $1/2$.  Hence, 
for $n\ge r + 1$, all but a $n^{-3}$-fraction of observation 
patterns with $m$ observed entries yield a matrix completion 
problem that is generic and uniquely completable, which is 
the matrix completion analogue of global rigidity 
(see \cite{kiraly-matcomp, singer-matcomp}).

We do not know that unique completability is a generic property,
even for symmetric matrices, but the existence of a generic 
rank $r$ matrix $M$ with a uniquely completable matrix completion 
problem for some pattern does imply that the observation 
pattern is finitely completable, which is a generic property 
\cite{kiraly-matcomp, singer-matcomp}.
\end{proof}
We notice that, although we are not dealing 
with symmetrically chosen observation patterns, 
since we assume the underlying matrix is 
symmetric, the probability of finite 
completability is not changed when making this 
assumption.  Similarly, the probability of a 
pattern being generically finitely completable 
can only go up if we deterministically observe 
the diagonal and then uniformly sample $m$ other 
entries.

\subsection{Rigidity in pseduo-Euclidean spaces}
A $d$-dimensional \emph{pseudo-Euclidean}  
space  $\MM^d_s$ is $\RR^d$ equipped with a bilinear form 
\[
    \beta_s(x,y) =  -\gamma_1\delta_1 - \cdots - \gamma_s\delta_s 
        + \gamma_{s+1}\delta_{s+1} + \cdots + \gamma_d\delta_d
\]
where $x = (\gamma_i)$ and $y = (\delta_i)$.  We can use 
$\beta_s$ to measure length and define local and global 
rigidity similarly to the Euclidean case (see 
\cite{GT-pseudo,whiteley-pseudo} for details).
\begin{thm}[\cite{whiteley-pseudo}]\label{thm: rigidity eqv}
For every dimension $d\ge 1$ and $d$-dimensional pseudo-Euclidean space $\MM_s^d$, 
local $d$-rigidity is a generic property.  Moreover,
a graph $G$ is locally $d$-rigid in $\MM_s^d$ 
if and only if $G$ is locally $d$-rigid in Euclidean
space.
\end{thm}

Meanwhile, as observed by Gortler and Thurston in 
\cite{GT-pseudo}, 
any symmetric $n\times n$ 
matrix of rank $d+1$ and signature 
$(d + 1 - s,n-d-1,s)$ arises as the 
Gram matrix of a configuration of 
$n$ vectors in a pseudo-Euclidean space with 
$\beta_s$ as the bilinear form.
Translating language slightly, we have:
\begin{thm}[\cite{GT-pseudo}]\label{thm: pseudo to matcomp}
Let $G$ be a graph on $n$ vertices.  Then 
$G$ is locally $d$-rigid in a pseudo-Euclidean 
space $\MM_s^d$ if and only if the associated 
symmetric set of observed entries, plus the diagonal, 
is generically finitely completable for rank $d+1$.
\end{thm}

\subsection{Putting things together}
We are now in a position to prove 
Theorem \ref{thm: rand rigid 2}.
From Theorem \ref{thm: gen matcomp random}, for 
each rank $d+1$, there is a constant $C_{d+1}$
so that, if the diagonals and $m\ge C_{d+1}n\operatorname{polylog}(n)$ uniformly selected 
entries of a generic symmetric matrix are observed, the 
resulting pattern is, whp, finitely completable.
Theorem \ref{thm: pseudo to matcomp} then implies that
the graph arising from symmetrising the observed entries 
is generically $d$-rigid in a pseudo-Euclidean space, and
then, by Theorem \ref{thm: rigidity eqv} generically $d$-rigid.
\end{appendix}

\begin{acks}[Acknowledgments]
This paper arose as part of the Fields Institute 
Thematic Program on Geometric constraint systems, 
framework rigidity, and distance geometry. LT thanks 
Bill Jackson for discussions about globally rigid 
$d$-circuits.
\end{acks}

\begin{funding}
DIB was partially supported by a Mathematical Sciences
Postdoctoral Research Fellowship from the US NSF, grant DMS-1802902.
SD was partially supported by the Austrian Science Fund (FWF): P31888.
AN was partially supported by the Heilbronn Institute for
Mathematical Research. SJG was partially supported by US NSF grant 
DMS-1564473. MS was partially supported by US
NSF grant DMS-1564480 and US NSF grant DMS-1563234. 

\end{funding}


\bibliographystyle{imsart-number} 
\bibliography{mltRigid}       
\end{document}